\documentclass{article}
\usepackage[utf8]{inputenc}
\usepackage{amsthm}
\usepackage{amsmath}
\usepackage{comment}
\usepackage{amsfonts}
\usepackage{color}
\usepackage{latexsym}
\usepackage{geometry}
\usepackage{enumerate}
\usepackage[shortlabels]{enumitem}
\usepackage{csquotes}
\usepackage{graphicx}
\usepackage{comment}
\usepackage{appendix}
\usepackage{hyperref}
\usepackage{esint}
\usepackage{bm}
\usepackage{amssymb}

\hypersetup{
    colorlinks=true,
    linkcolor=blue,
    filecolor=magenta,      
    urlcolor=cyan,
}

\emergencystretch=\maxdimen
\def \cR {\mathcal{R}}
\def \cX {\mathcal{X}}

\def \bF {\mathbb{F}}
\def \fC {\mathfrak{C}}
\def \cS {\mathcal{S}}

\def \ep {\varepsilon}
\def \bN {\mathbb{N}}

\def \fg {\mathfrak{g}}

\DeclareMathOperator{\Rm}{Rm}
\DeclareMathOperator{\Ric}{Ric}

\DeclareMathOperator{\Var}{Var}
\DeclareMathOperator{\AVR}{AVR}
\DeclareMathOperator{\loc}{loc}

\makeatletter
\newcommand*{\rom}[1]{\rm {\expandafter\@slowromancap\romannumeral #1@}}
\makeatother

\def\XXint#1#2#3{{\setbox0=\hbox{$#1{#2#3}{\int}$ }
\vcenter{\hbox{$#2#3$ }}\kern-.6\wd0}}

\protected\def\vts{%
  \ifmmode
    \mskip0.5\thinmuskip
  \else
    \ifhmode
      \kern0.08334em
    \fi
  \fi
}

\numberwithin{equation}{section}
\newtheorem{Theorem}{Theorem}[section]
\newtheorem{Proposition}[Theorem]{Proposition}
\newtheorem{Lemma}[Theorem]{Lemma}
\newtheorem{Corollary}[Theorem]{Corollary}
\newtheorem{Conjecture}[Theorem]{Conjecture}
\theoremstyle{definition}
\newtheorem{Definition}[Theorem]{Definition}

\title{Ancient Ricci flows with nonnegative Ricci curvature}
\author{Yuxing Deng\footnote{Yuxing Deng's research is  supported by  National Key R$\&$D Program of China 2022YFA1007600 and National Natural Science Foundation of China NSFC12571059.}, Ganqi Wang, Yongjia Zhang\footnote{Yongjia Zhang's research is  supported by National Natural Science Foundation of China NSFC12301076.}}
\date{}

\begin{document}
\maketitle
\begin{abstract}

In this paper, we study the asymptotic geometry of a noncollapsed ancient Ricci flow with nonnegative Ricci curvature via its tangent flow at infinity --- a noncollapsed $\mathbb{F}$-limit metric soliton \cite{Bam23,CMZ23}. We first prove some estimates for noncollapsed $\mathbb{F}$-limit metric solitons with nonnegative Ricci curvature, and then obtain two dichotomy theorems for ancient Ricci flows.  In particular, we show that: (1)  for a noncollapsed ancient Ricci flow with nonnegative Ricci curvature, either its asymptotic volume ratio is always zero, or every tangent flow at infinity is a Ricci flat cone; (2) for a noncollapsed ancient Ricci flow with positively pinched Ricci curvature ($\Ric\ge \varepsilon R g$), either  it is compact, or every tangent flow at infinity is a Ricci flat cone.
\end{abstract}
%\tableofcontents
\section{Introduction}

In \cite{Per02}, Perelman proved the existence of an asymptotic shrinker for every $\kappa$-solution; by $\kappa$-solution we mean noncollapsed ancient solution with nonnegative curvature operator. This result sheds much light on the asymptotic geometry for $\kappa$-solutions, and is therefore an important part of the proof of the geometrization conjecture \cite{Ham93,Per02,Per03a,Per03b}, since all three-dimensional singularity models are $\kappa$-solutions. For instance, combining the existence of asymptotic shrinker and the classification of three-dimensional Ricci shrinkers, Perelman showed that there is a universal noncollapsing constant $\kappa_0>0$ for all $\kappa$-solutions. There are many other examples where asymptotic shrinkers are implemented in the study of ancient solutions; see \cite{Nab10,CZ11,Zhang20,CMZ25}, to name but a few.

For general noncollapsed ancient solutions, asymptotic shrinkers, unfortunately, do not necessarily exist. Its existence is known to be true either in the case of Type I curvature bound \cite{Nab10}, or in the case where there holds Hamilton's trace Harnack estimate, such as the PIC2 case \cite{Bre09}. In general, because of the lack of the ``bounded curvature at bounded distance condition'' for the blow-down sequence, the asymptotic limit for an ancient Ricci flow either does not exist, or, even if it existed, may not be smooth.

Bamler's works \cite{Bam20a,Bam20b,Bam23} provide many new ideas and new techniques for sequences of Ricci flows that do not necessarily converge smoothly. Viewing a Ricci flow coupled with a conjugate heat kernel as an evolving metric measure space (called a metric flow pair, see \S 2), he defined a weak notion of convergence called the $\mathbb{F}$-convergence for a sequence of metric flow pairs, which, roughly speaking, is tantamount to the Gromov-Wasserstein convergence for almost every time-slice. It is proved in \cite{Bam23} that every sequence of $H$-concentrated metric flow pairs, after passing to a subsequence, converges in the $\mathbb{F}$-sense. More details can be found in \S 2. In a recent work of Fang-Li \cite{FL25}, a new notion of weak convergence is introduced.

If a sequence of metric flow pairs (constructed with Ricci flow of the same dimension) is noncollapsed, namely, admits a uniform bound for the Nash entropy, then its $\mathbb{F}$-limit is an almost-everywhere-smooth metric flow pair, with its singular part  a set of space-time Minkowski codimension no less than four. Furthermore, the convergence is also smooth on the regular part of the limit flow. In addition, if the Nash entropies of the sequence converge to a constant, then the $\mathbb{F}$-limit is a self-similar metric flow pair called a \emph{noncollpased $\mathbb{F}$-limit metric soliton}; see \S 2 for more details. This object is exactly the asymptotic limit for general noncollapsed ancient solutions corresponding to Perelman's asymptotic shrinker.

Noncollpased $\mathbb{F}$-limit metric soliton is understood to be a natural generalization of smooth Ricci shrinker. We briefly summarize some of its properties proved in \cite{Bam20b}. Sometimes, because of its self-similarity, a noncollapsed $\mathbb{F}$-limit metric solitons is not distinguished from its model, a metric mesure space $(X,d,\nu)$, which is identified with the $t=-1$ time-slice of the self-similar metric flow pair. Furthermore, the model is fully determined by its regular part $(\mathcal{R}_X,\fg,f_0)$, since the metric completion of $(\mathcal{R}_X,d_{\fg})$ is $(X,d)$ and the singular part is a null-set, i.e.
\begin{align*}
    \nu(X\setminus\mathcal{R}_X)=0,
\end{align*}
with Minkowski codimension no less that four. Here, $f_0$ is a smooth function on $\mathcal{R}_X$ called the potential function, and $\fg$ and $f_0$ satisfy the shrinker equation on $\mathcal{R}_X$:
\begin{align*}
    \Ric_{\fg}+\nabla^2_{\fg} f_0=\frac{1}{2}\fg\quad\text{ and }\quad  d\nu=  (4\pi)^{-\frac{n}{2}}e^{-f_0}d\fg\quad
     \quad \text{ on }\qquad \mathcal{R}_X,
\end{align*}

Because of its importance as generalized singularity models, \cite{CMZ23} studied systematically noncollpased $\mathbb{F}$-limit metric solitons, to which many classical results for Ricci shrinkers are extended. For instance, the potential function $f_0$ has almost quadratic growth from a point $x_0\in\mathcal{R}_X$. Such point is called a \emph{center} of the soliton, and henceforth will always be denoted by $x_0$; see \S 2. Furthermore, a global Sobolev inequality, a quadratic lower bound estimate for the scalar curvature, a volume growth lower bound, and a local gap theorem for large balls around a center are proved in \cite{CMZ23}.

To summarize the points we have made by far, let us recall the following facts: Perelman's asymptotic shrinker is a powerful tool for the study of the asymptotic geometry of $\kappa$-solutions; general noncollapsed ancient solutions do not have asymptotic shrinkers, but their tangent flows at infinity are noncollpased $\mathbb{F}$-limit metric solitons; the convergence of the blow-down sequence of an ancient solution to its tangent flow at infinity is smooth on the smooth part of the limit flow; the tangent flow at infinity is almost everywhere a smooth shrinker. It is therefore clear that, by studying noncollpased $\mathbb{F}$-limit metric solitons, especially its smooth part, we can obtain some results concerning the asymptotic geometry of noncollapsed ancient solutions. The main results of this paper are in this fashion.

Before introducing our results, a few clarifications need to be made for our terminologies. Our notion of noncollapsedness is somewhat different from Perelman's, but is identical to Bamler's, namely, our noncollapsing condition is defined in terms of the boundedness of the Nash entropy. By saying that an ancient Ricci flow $(M^n,g_t)_{t\in(-\infty,0]}$ is noncollapsed, we mean that for some $(x,t)\in M\times (-\infty,0]$ and $Y>0$, we have
\begin{align*}
    \mathcal{N}_{x,t}(\tau)\ge -Y \quad \text{ for all }\quad \tau>0.
\end{align*}
By saying that a sequence $\{(M^i,g^i_{t},x_i)_{t\in[-T_i,0]}\}_{i=1}^\infty$ of $n$-dimensional Ricci flows is noncollapsed, we mean that for some $\tau>0$ and $Y<\infty$, we have
\begin{align*}
    \mathcal{N}_{x_i,0}(\tau)\ge -Y\quad \text{ for all }\quad i.
\end{align*}
A \emph{noncollapsed} $\mathbb{F}$-limit metric soliton is therefore the $\mathbb{F}$-limit of a \emph{noncollapsed} sequence of Ricci flows.

We first introduce our results on noncollpased $\mathbb{F}$-limit metric solitons with nonnegative Ricci curvature, and then we present their applications to ancient Ricci flows. For the definitions that appear in the statements of the theorems, refer to \S 2. Our first result shows that on the smooth part of a non-Ricci-flat metric soliton with nonnegative Ricci curvature, the scalar curvature must be bounded from below by a positive constant; this result is analogous to \cite[Proposition 1.1]{Ni05}.

\begin{Theorem}\label{thmR>}
    Let $(X,d,\nu)$ be (the model of) a noncollapsed $\mathbb{F}$-limit metric soliton and $(\mathcal{R}_X,\fg,f_0)$ its regular part. Assume that the Ricci curvature is non-negative on $\mathcal{R}_X$ and does not vanish everywhere. Then there exists a positive number $\delta = \delta(X)\in(0,1]$ such that
 $$R(x)\ge\delta,\quad\text{for all}\quad x\in \mathcal{R}_X.$$
\end{Theorem}

In spite of the similarity of the conclusion of the above result to \cite[Proposition 1.1]{Ni05}, our proof is somewhat different, since we cannot estimate the Ricci curvature along a geodesic by the second variational formula;    the regular part of a metric soliton is not necessarily convex, and any geodesic may possibly exceed the regular part. Therefore, we estimate the scalar curvature along integral curves of $\nabla f_0$ instead.

Our next result on metric soliton is a volume estimate of the distance ball $B_{\fg}(x_0,A)\cap\cR_X$, which is not necessarily a geodesic ball since $(\mathcal{R}_X,\fg)$ is generally not complete. For complete smooth Ricci shrinkers, Cao and Zhou \cite{CZ10} showed  $|B(x_0,A)|\le CA^n$, where $C$ is a constant depending on the soliton; an improvement was given by Haslhofer-M\"uller \cite{HM11}, showing that if $x_0$ is a minimum point of $f$, then $C$ is a dimensional constant; a recent work of Li-Wang \cite{LW24} further improved Cao-Zhou's estimate, showing that $C$ is a dimensional constant for any $x_0$; if the scalar curvature of the shrinker is bounded from below by a positive constant, namely, $R\ge \delta>0$, then Zhang \cite{Zhang11} gave a sharp volume growth estimate of the form $|B(x_0,A)|\le CA^{n-2\delta}$, where $C$ is a constant depending on the soliton, and the sharpness of the estimate is seen in the spherical and cylindrical cases. We generalize \cite[Theorem 1.2]{Zhang11} to distant balls on the regular part of a metric soliton.
\begin{Theorem}\label{thmavr}
   Let $(X,d,\nu)$ be (the model of) a noncollapsed $\mathbb{F}$-limit metric soliton and $(\mathcal{R}_X,\fg,f_0)$ its regular part.  Assume the scalar curvature is bounded from below by some $\delta>0$ on $\mathcal{R}_X$. Then we have 
    $$|B_{\fg}(x_0,A)\cap\cR_X|\le CA^{n-2\delta}\quad \text{ for all }\quad A\ge 1,$$
    where $x_0$ is a center of the soliton, and $C$ is a constant depending on the soliton.
\end{Theorem}

Next, we shall consider the asymptotic volume ratio ($\AVR$) of a metric soliton. For a complete Riemannian manifold $(M^n,g)$, the $\AVR$ is defined to be
\begin{align*}
    \AVR:=\lim_{A\to+\infty}\frac{|B_{g}(x,A)|}{\omega_n A^n},
\end{align*}
where $x_0$ is a fixed point on $M$. Note that the $\AVR$ may not exist for a manifold. But for manifolds with nonnegative Ricci curvature, it always exists and is independent of $x$. A result by Chow-Lu-Yang \cite{CLY12} shows that the $\AVR$ is well-defined for every Ricci shrinker, and we prove the same result for noncollapsed $\mathbb{F}$-limit metric solitons.

\begin{Theorem}\label{thm existence of AVR}
    Let $(X,d,\nu)$ be (the model of) a noncollapsed $\mathbb{F}$-limit metric soliton, $(\mathcal{R}_X,\fg,f_0)$ be its regular part, and $x_0$ be a center of the metric soliton.  Then the limit 
    $$\lim_{A\to+\infty}\frac{|B_{\fg}(x_0,A)\cap\cR_{X}|}{\omega_n A^n}\in[0,C(n)e^{W}]$$
    exists, where $\omega_n$ denotes the volume of the unit Euclidean ball and $W\le 0$ is the soliton entropy. We define the limit as the $\AVR$ of a metric soliton.
\end{Theorem}

For a Ricci flow, the $\AVR$ is defined time-wise, namely, it is a function of time, should it exist. Perelman proved that the $\AVR$ is always zero for a non-flat $\kappa$-solution; this is an important element in his proof of the $\kappa$-compactness theorem \cite[\S 10]{Per02}. Therefore, it is also interesting to study the $\AVR$ in more general cases. 
If we replace $\Rm\ge0$ in \cite[11.4]{Per02} by $\Ric\ge0$, we have the following result.
\begin{Theorem}\label{thm tangent flow character1}
    Let $(M^n,g_t)_{t\in(-\infty,0]}$ be a noncollapsed ancient Ricci flow with bounded curvature on each compact time interval. Suppose $\Ric\ge0$ on $M\times(-\infty,0].$ Then either 
    \begin{enumerate}[(i)]
        \item every tangent flow at infinity is a Ricci-flat cone, or
        \item  $\AVR_{g_t}\equiv 0$ for all $t\le 0$.
    \end{enumerate}
\end{Theorem}

At last, we consider noncollapsed ancient Ricci flows with positively pinched Ricci curvature ($\Ric\ge \varepsilon Rg$). This study is inspired by a problem proposed by Hamilton (see also \cite{Ni05}) --- is it true that a complete manifold with positively pinched Ricci curvature always compact? Hamilton asked this question in dimension three, since an affirmative answer will  simplify Hamilton's classification of three-dimensional manifolds with positive Ricci curvature \cite{Ham82}. However, it was solved in dimension three in \cite{LT25} by applying Hamilton's Ricci flow, and it remains open in higher dimensional cases.

On the other hand, Hamilton \cite{Ham94} solved an extrinsic counterpart of the question above, namely, every complete hypersurface with positively pinched second fundamental form must be closed. Ni \cite{Ni25} gave an alternative proof for Hamilton's theorem, and also a relatively comprehensive survey of this problem. It is known that any Ricci shrinker with positively pinched Ricci curvature must be compact. This is nothing but a corollary of \cite[Proposition 1.1]{Ni05}. For expanding and steady cases, if, in addition, the potential function admits a critical point, then every soliton with positively pinched Ricci curvature must be Ricci flat; see \cite{DZ15}. A noncollapsed ancient solution with a smooth asymptotic soliton is clearly closed if it has positively pinched Ricci curvature. For the general case, we obtain the following result.

\begin{Theorem}\label{thm tangent flow character2}
    Let $(M^n,g_t)_{t\in(-\infty,0]}$ be a noncollapsed ancient Ricci flow with bounded curvature within each compact time interval. Suppose $$\Ric\ge\varepsilon Rg \quad \text{ on }\quad M\times(-\infty,0]$$ for some $\varepsilon>0$. Then either 
    \begin{enumerate}[(i)]
        \item every tangent flow at infinity is a Ricci-flat cone, or
        \item the ancient Ricci flow is compact.
    \end{enumerate}
\end{Theorem}

Finally, in view of the dichotomy  in Theorem \ref{thm tangent flow character1} and Theorem \ref{thm tangent flow character2}, we propose the following  conjecture:

\begin{Conjecture}
    Let $(M,g_t)_{t\in(-\infty,0]}$ be a noncollapsed ancient Ricci flow with bounded curvature within each compact time interval. Suppose $\Ric_{g_t}\ge 0$ on $M\times (-\infty,0]$ and every tangent flow at infinity is a Ricci flat cone. Then $g_t$ is a static Ricci flat metric with positive $\AVR$.
\end{Conjecture}

This paper is organized as follows. In \S 2, we introduce some basic notions and results in the literature, while omitting some lengthy definitions in \cite{Bam23}. In \S 3, we prove Theorem \ref{thmR>}, Theorem \ref{thmavr}, and Theorem \ref{thm existence of AVR}. In \S 4, we prove Theorem  \ref{thm tangent flow character1}. In \S 5, we prove Theorem  \ref{thm tangent flow character2}.
\\

\emph{Acknowledgment.} The last author would like to thank Professor Lei Ni for some inspiring discussions about Hamilton's problem. He would also like to thank Professor Man-Chun Lee and Professor Pak-Yeung Chan for some helpful conversations.

%--------------------------------------Section2----------------------------------------------------------

\section{Preliminaries}

Bamler's theory of $\mathbb{F}$-convergence and $\mathbb{F}$-compactness \cite{Bam20b, Bam23} is our essential technique, so we feel it necessary to introduce the notions and results applied by us in this section. However, because of the excessive lengths, we have to keep our words brief and refer the readers to where they truly belong. In particular, we shall not introduce the results in \cite{Bam20a}, such as $H_n$-concentration, the gradient estimate for the heat equation, etc.

\subsection{$\mathbb{F}$-convergence and $\mathbb{F}$-compactness}

Bamler's $\mathbb{F}$-convergence is defined for sequences of metric flow pairs. A metric flow pair $(\mathcal{X},(\mu_t)_{t\in I})$ can be viewed as an evolving metric measure space, namely, for each $t\in I$, $\mathcal{X}_t$ is a complete metric space and $\mu_t$ is a probability measure on $\mathcal{X}_t$. For the exact definition of metric flow pair, see \cite[Definition 3.1, Definition 5.1]{Bam23}. The notion of $\mathbb{F}$-convergence, roughly speaking, is the Gromov-Wasserstein convergence for metric measure spaces at almost every time; the exact definition is found in \cite[Definition 5.5, Definition 5.7]{Bam23}.

Bamler also proved a compactness result for $H$-concentrated sequences of metric flow pairs: from a sequence of $H$-concentrated metric flow pairs, one can always extract a subsequence converging in the $\mathbb{F}$-sense to an $H$-concentrated metric flow pair; see \cite[Theorem 7.4]{Bam23}.

What is of our interest the most is the convergence of a sequence of metric flow pairs made of smooth Ricci flows. Let $(M^i,g^i_{t})_{(-T_i,0]}$ be a sequence of Ricci flows with bounded curvature within each compact time interval. Let $x_i\in M^i$ be a fixed base point for each $i$. Then, according to \cite{Bam20a}, $\big\{\big((M^i,g^i_t)_{t\in(-T_i,0]},(\nu^i_{x_i,0\,;\,t})_{t\in(-T_i,0]}\big)\big\}_{i=1}^\infty$ is a sequence of $H_n$-concentrated metric flow pairs, where 
\begin{align*}
    d\nu_{x,t\,;\, s}:= K(x,t\,;\,\cdot,s)\,dg_s,\quad s\le t.
\end{align*}
always stand for the conjugate heat kernel based at $(x,t)$ and 
$$H_n=\frac{(n-1)\pi^2}{2}+4.$$
By \cite[Theorem 7.4]{Bam23}, we can find a (not relabeled) subsequence, such that
\begin{align}\label{eq:limiting sequence}
    \left((M^i,g^i_t)_{t\in(-T_i,0]},(\nu^i_{x_i,0;t})_{t\in(-T_i,0]}\right)\xrightarrow[i\to\infty]{\mathbb{F},\ \mathfrak{C}, \ J}(\mathcal{X},\nu_t),
\end{align}
where $\mathfrak{C}$ stands for correspondence (see \S 2.2 and \cite[Definition 6.1]{Bam23}),  $(\mathcal{X},\nu_t)$ is an $H_n$-concentrated metric flow pair defined on $(-T,0]$, $T:=\limsup_{i\to\infty} T_i\in(0,+\infty]$, $J$ is a prescribed finite subset of $(-T,0]$ such that the convergence is time-wise on $J$, and $\nu_t$ is a conjugate heat flow satisfying 
$$\operatorname{Var}(\nu_t)\le H_n|t|.$$
Here $\operatorname{Var}$ is the variance. 

If, in addition, the sequence in \eqref{eq:limiting sequence} is noncollapsed, namely, 
\begin{align}\label{eq:noncollapsing assumption}
    \mathcal{N}_{x_i,0}(\tau)\ge -Y\quad \text{ for each }\ i,
\end{align}
for some $\tau>0$ and $Y<+\infty$, then some partial regularity results were proved by Bamler \cite{Bam20b} for the limit flow:

\begin{Theorem}[Bamler's partial regularity result {\cite[Theorem 2.4, Theorem 2.5]{Bam20b}}]
    Under the assumption \eqref{eq:noncollapsing assumption}, the limit flow pair $(\mathcal{X},(\nu_t)_{t\in(-T,0)})$ from \eqref{eq:limiting sequence} admits a decomposition $\mathcal{X}=\mathcal{R} \sqcup \mathcal{S}$   satisfying the following properties:
    \begin{enumerate} [(a)]
        \item The regular part $\mathcal{R}$ is a smooth Ricci flow space-time: $\mathcal{R}$ is locally space-time product, and on each slice $\mathcal{R}_t$, there is a metric $\fg_t$, such that $\fg_t$ satisfies the Ricci flow equation. Furthermore, $d\nu_t=u_td\fg_t$ on $\mathcal{R}_t$, where $u$ is a positive solution to the conjugate heat equation on $\mathcal{R}$;
        \item $\mathcal{S}$ is a set of measure zero for each $t$, and the space-time Minkowski codimension of $\mathcal{S}$ is no smaller than four;
        \item For each $t$, the metric completion of $(\mathcal{R}_t,\fg_t)$ is the metric space $\mathcal{X}_t$;
        \item The convergence is smooth on $\mathcal{R}$ (see \S 2.3 for more details).
    \end{enumerate}
    Here (and always) $\mathcal{X}_t$, $\mathcal{R}_t$, and $\mathcal{S}_t$ stand for time-slices.
\end{Theorem}

\subsection{Convergence of points in $\mathbb{F}$-convergence}

When considering the Cheeger-Gromov convergence, it is easy to compare points in different manifolds of the sequence, since such kind of convergence is defined via diffeomorphisms, which can be used to send these points to the limit manifold. In the case of the $\mathbb{F}$-convergence, we will resort to a tool of Bamler called correspondence for such purpose.

The exact definition of correspondence is found in \cite[Definition 5.4]{Bam23}. Roughly speaking, it is an evolving family of large metric spaces to which we embed all metric flow pairs under consideration and wherein we compare their closeness. This idea was first introduced by Gromov to define the Gromov-Hausdorff distance.

It is important to note that the $\mathbb{F}$-convergence is equivalent to the $\mathbb{F}$-convergence within a correspondence \cite[Theorem 6.6]{Bam23}. Let us consider the following convergence of metric flow pairs (for the sake of simplicity, we omit for now the finite set $J$ on which the convergence is time-wise, and also ignore the fact that the flows may not be fully defined on intervals):
\begin{align*}
(\mathcal{X}^i,\nu_t^i)\xrightarrow{\mathbb{F},\ \mathfrak{C}} (\mathcal{X},\nu_t),
\end{align*}
where $\mathcal{X}^i$ and $\mathcal{X}$ are defined over $(-T,0]$, and $\mathfrak{C}=\big((Z_t,d^Z_t),\{\varphi^i_t\}_{i=1}^\infty,\varphi_t\big)_{t\in(-T,0]}$ is a correspondence, such that for each $t\in(-T,0]$
\begin{align*}
    \varphi^i_t: \mathcal{X}_t^i \to (Z_t,d^Z_t)\quad i\in\mathbb{N},\qquad  \varphi_t: \mathcal{X}_t\to (Z_t,d^Z_t)
\end{align*}
are isometric embeddings.

Now there are two kinds of different convergences of points. Suppose $x_i\in\mathcal{X}_{t_i}$ and $x\in\mathcal{X}_t$, then we say $x_i$ converges to $x$ within $\mathfrak{C}$, and write
\begin{align*}
    x_i\xrightarrow[i\to\infty]{\mathfrak{C}} x,
\end{align*}
if
\begin{align*}
   t_i\to t\qquad \text{ and }\qquad  (\nu_{x_i\,;\,s}^i)_{s<t_i} \xrightarrow[i\to\infty]{\mathfrak{C}} (\nu_{x\,;\,s})_{s<t}.
\end{align*}
The convergence of conjugate heat flow within correspondence is defined in \cite[Definition 6.7]{Bam23}.
If, on the other hand, $t_i=t$ for all $i$ and 
\begin{align*}
    \varphi^i_t(x_i)\to \varphi_t(x)
\end{align*}
in the metric space $(Z_t,d_t^Z)$, then we say $x_i$ strictly converges to $x$. It is known that for sequences of points, the strict convergence indeed implies the convergence within a correspondence \cite[Theorem 6.13]{Bam23}.

The following theorem is a restatement of {\cite[Theorem 6.20]{Bam23}} in the case of \eqref{eq:limiting sequence}.

\begin{Theorem}[Compactness of sequence of points {\cite[Theorem 6.20]{Bam23}}]\label{Thm:compactness of points}
    Consider the convergence \eqref{eq:limiting sequence}. We denote by $\mathfrak{C}=\big((Z_t,d^Z_t),\{\varphi^i_t\}_{i=1}^\infty,\varphi_t\big)_{t\in(-T,0]}$ the correspondence. Let $(y_i,t_i)\in M^i\times (-T_i,0]$ be a sequence of points with $\lim_{i\to\infty}t_i=t_0\in(-T,0]$. Suppose
    \begin{align*}
        d^{g^i_{t_i}}_{W_1}(\delta_{y_i},\nu^i_{x_i,0\,;\,t_i})\le D\quad \text{ for each }i,
    \end{align*}
    where $D<+\infty$ is a positive constant. Then there is a conjugate heat flow $\mu_t$ defined for almost every $t\in(-T,t_0)$ and defined fully on $(-T,t_0)\cap J$, such that
    \begin{align*}
        \lim_{t\nearrow t_0} \Var(\mu_t)=0
    \end{align*}
    and, after passing to a subsequence
    \begin{align*}
        \nu_{y_i,t_i\,;\,t}\xrightarrow[i\to\infty]{\mathfrak{C}, J} \mu_t.
    \end{align*}
    In particular, we have
    \begin{align}\label{eq:convergence of CHF}
        \lim_{i\to\infty} d^{Z_t}_{W_1}\Big((\varphi^i_t)_*\nu_{y_i,t_i\,;\,t},(\varphi_t)_*\mu_t\Big)=0\quad \text{ for each }t\in J\cap(-T,t_0).
    \end{align}
    Here $J\subset (-T,0]$ is a finite set where the convergence \eqref{eq:limiting sequence} is time-wise; \eqref{eq:convergence of CHF} is a consequence of \cite[Definition 6.7]{Bam23}
\end{Theorem}

\subsection{Local smooth convergence}

We have already mentioned the fact that, if \eqref{eq:limiting sequence} and \eqref{eq:noncollapsing assumption} both hold, then the convergence \eqref{eq:limiting sequence} is locally smooth on the regular part of the limit flow. For convenience of application, we shall summarize some basic properties of local smooth convergence. The following theorem is nothing but a combination of  \cite[Theorem 9.21]{Bam23} and \cite[Theorem 2.5]{Bam20b}.

\begin{Theorem}[{\cite[Theorem 9.21]{Bam23} and \cite[Theorem 2.5]{Bam20b}}]\label{Thm_smooth_convergence}
Suppose \eqref{eq:limiting sequence} and \eqref{eq:noncollapsing assumption} both hold, then we can find an increasing sequence $U_1 \subset U_2 \subset \ldots \subset \mathcal{R}$ of open subsets with $\bigcup_{i=1}^\infty U_i = \cR$, open subsets $V_i \subset M^i\times(- T_i,0]$, time-preserving diffeomorphisms $\psi_i : U_i \to V_i$ and a sequence $\ep_i \to 0$ such that the following hold:
\begin{enumerate}[label=(\alph*)]
\item \label{Thm_smooth_convergence_a} We have
\begin{align*}
 \Vert \psi_i^* g^i - \fg \Vert_{C^{[\ep_i^{-1}]} ( U_i)} & \leq \ep_i, \\
  \Vert  u^i \circ \psi_i - u \Vert_{C^{[\ep_i^{-1}]} ( U_i)} &\leq \ep_i, 
\end{align*}
where $d\nu^i_{x_i,0\,;\,t}=u^i(\cdot,t)\,dg^i_t$, $d\nu_{t}=u(\cdot,t)\,d\mathfrak{g}_t$.
\item \label{Thm_smooth_convergence_b} Let $x \in \cR$ and $(x_i,t_i) \in M^i\times(-T_i,0]$.
Then $(x_i,t_i) \to x$ within $\fC$ if and only if $(x_i,t_i) \in V_i \subset M^i\times(-T_i,0]$ for large $i$ and $\psi_i^{-1} (x_i,t_i) \to x$ in $\cR$.
\item \label{Thm_smooth_convergence_c} If the convergence \eqref{eq:limiting sequence} is time-wise at some time $t \in (-T,0]$ for some subsequence, then for any compact subset $K \subset \cR_t$ and for the same subsequence
\[ \sup_{x \in K \cap U_i} d^Z_t (\varphi^i_t (\psi_i(x)), \varphi_t (x) )  \longrightarrow 0. \]
\item \label{Thm_smooth_convergence_d} Consider a sequence of conjugate heat flows $(\tilde\nu_t^i)_{t < t_0}$ on $M^i\times(-T_i,0]$, $i \in \bN $, for $t_0 \le0$ such that
\[ (\tilde\nu_t^i)_{ t < t_0} \xrightarrow[i \to \infty]{\quad \fC \quad} (\tilde\nu_t)_{ t < t_0}. \]
Write $d\tilde\nu^i_t = \tilde v^i_t \, dg^i_t$  for $i \in \bN$ and $d\tilde\nu_t = \tilde v_t \, d\mathfrak{g}_t$.
Then on $\cR$
\[    \tilde v^i \circ \psi_i  \xrightarrow[i \to \infty]{\quad C^\infty_{\loc} \quad} \tilde v . \]
\end{enumerate}
\end{Theorem}

We also need the following convergence of parabolic neighborhoods. Consider the Ricci flow space-time $\mathcal{R}$ and a point $x\in \mathcal{R}_t$. For $D,T^+,T^->0$, $P^o(x;D,-T^-,T^+)$ is defined to be the union of all world-lines at each $y\in B_{\fg_t}(x,D)\cap\mathcal{R}_t$ within interval $(t-T^-,t+T^+)$. If for each $s\in(t-T^-,t+T^+)$, the set $P^o(x;D,-T^-,T^+)\cap \mathcal{R}_s$ is precompact in $\mathcal{R}_s$, then $P^o(x;D,-T^-,T^+)$ is called unscathed.

\begin{Theorem}[{\cite[Theorem 9.24]{Bam23}}]
    In the setting of Theorem \ref{Thm_smooth_convergence}, suppose $(y_i,t_i)$ converges to $y_0\in\mathcal{R}_{t_0}$ within $\mathfrak{C}$, where $t_0\in(-T,0)$. Assume for $D,-T^-,T^+>0$ and $C<\infty$, we have $|\Rm_{g_t^i}|\le C$ on $P_i:=P^o(y_i,t_i; D,-T^-,T^+)$ for each $i$. Then there is some $0<T^*\le T^+$, such that $P_\infty:=P^o(y_0; D,-T^-,T^*)$ is unscathed and $|\Rm|\le C$ on $P_\infty$. Furthermore, if $T^*<T^+$ and $t_0+T^*<0$, then no point in $P_\infty$ survives until or after $t_0+T^*$.
\end{Theorem}

\subsection{Noncollapsed $\mathbb{F}$-limit metric soliton}

If, in addition to \eqref{eq:limiting sequence} and \eqref{eq:noncollapsing assumption}, we also assume the Nash entropies converge time-wise to a constant function, namely,
\begin{align}\label{eq:entropy converging to const}
    \lim_{i\to\infty}\mathcal{N}_{x_i,0}(\tau)=W\quad \text{ for all }\quad \tau \in(0,T), 
\end{align}
where $W$ is a constant, then the limit flow $(\mathcal{X},\nu_t)$ is known to be a noncollapsed $\mathbb{F}$-limit metric soliton. $W$ is called the \emph{soliton entropy}. Let us recall some basic facts about such self-similar metric flows.

\subsubsection{Basic equations and regularity}

\begin{Theorem}[{\cite[Theorem 2.18, Theorem 15.69]{Bam20b}}]\label{thmbase}
    For a noncollapsed $\mathbb{F}$-limit metric soliton $(\mathcal{X},(\nu_{t})_{t\in(-T,0)}),$ we write $\tau=-t$, $d\nu_t=(4\pi\tau)^{-\tfrac{n}{2}}e^{-f}d\fg_t$ on $\mathcal{R}.$ 
    \begin{enumerate}[(a)]
        \item  On the regular part $\mathcal{R}$, $\nabla f$ is complete, and we have
    \begin{equation}\label{eq1}
        \Ric+\nabla^2f-\tfrac{1}{2\tau}\fg_t=0,\qquad -\tau\left(|\nabla f|^2+R\right)+f\equiv W,
    \end{equation}
where $W$ is the soliton entropy.
\item There is a metric measure space $(X,d, \nu)$, called the \emph{model of the metric soliton}, such that
$$\mathcal{X}_{<0}=X\times(-T,0),$$
and the following hold for all $t\in(-T,0)$:

\begin{enumerate}[(a)]
    \item $(\mathcal{X}_t,d_t)=(X\times\{t\},|t|^{1/2}d)$.
    \item $\mathcal{R}=\mathcal{R}_X\times(-T,0)$, where $\mathcal{R}_X$, the regular part of $X$, is a smooth manifold.
    \item $(\mathcal{R}_t,\fg_t)=(\mathcal{R}_X\times\{t\},|t|\fg)$, where $\fg$ is a Riemannian metric on $\mathcal{R}_X$.
    \item $\nu_t = \nu$.
\end{enumerate}
Moreover, there is a unique family of probability measures $(\nu'_{x;t})_{x\in X,\, t\ge 0}$ such that the tuple $(X,d,\nu,(\nu'_{x;t})_{x\in X,\, t\ge 0})$ is a model for $\big(\mathcal{X},(\nu_t)_{t\in(-T,0)}\big)$ in the sense of \cite[Definition 3.57]{Bam23}.

\item The singular part of the model $(X,d,\nu)$ is a null set with codimension no less than $4$, and $(X,d)$ is the metric completion of $(\mathcal{R}_X,\fg)$.

\item Writing $f_0:=f(\cdot,-1)$, then the tuple $(\mathcal{R}_X,\fg,f_0)$, called \emph{the regular part of} $(X,d,\nu)$, satisfies the shrinker equations
\begin{gather}\label{eq: basic-properties-2}
        \Ric+\nabla^2f_0=\frac{1}{2}\fg,\quad  |\nabla f_0|^2+R=f_0- W,
\\
        \nabla R =2\Ric(\nabla f_0). \label{eq:Ricci identity and f}
    \end{gather}
\item If $(\mathcal{R}_X,\fg)$ is non-Ricci-flat, then the scalar curvature is positive everywhere on $\mathcal{R}_X$. If $(\mathcal{R}_X,\fg)$ is Ricci-flat, then $(X,d)$ is a metric cone.
\end{enumerate}

\end{Theorem}

\subsubsection{The center and the quadratic growth of the potential function}

According to \cite[\S 4.1]{CMZ23}, we can find a point $x_0\in\mathcal{R}_X=\mathcal{R}_{-1}$, such that
$$\Var(\nu_t,\delta_{x_0(t)})\le H_n|t|\quad \text{ for all } t\in(-T,0).$$
Here $x_0(t)$ is the world-line of $x_0$ in $\mathcal{R}$, which exists for all $t$ due to the self-similarity. Such $x_0$ is called a \emph{center} of the metric soliton. Henceforth, we shall always use the notation $x_0$ to represent the center.

\begin{Theorem}[{\cite[Lemma 7.12]{FL25}}]\label{thm<f<}
    Let $(X,d,\nu)$ be the model of an n-dimensional non-collapsed $\mathbb{F}$-limit metric soliton and $(\mathcal{R}_X,\fg,f_0)$ the regular part of the model. Let $x_0\in\mathcal{R}_X$ be a center of the soliton. Then, for any $\varepsilon>0,$ we have
    \begin{equation*}
        \tfrac{1}{4+\varepsilon}d^2_{\fg}(x_0,x)-C(n,\varepsilon)\le f_0(x)-W\le\tfrac{1}{4}(d_{\fg}(x_0,x)+C(n))^2,
    \end{equation*}
    where $W$ is the soliton entropy.
\end{Theorem}

\subsubsection{Volume and scalar curvature of sub-level-set}

Consider $(\mathcal{R}_X,\fg,f_0)$ as mentioned above, with $x_0\in\mathcal{R}_X$ a center. Define $\rho:=2\sqrt{f_0-W}\ge0,$ and by Theorem \ref{thm<f<} we have 
\begin{equation}\label{ineq4}
    \sqrt{\tfrac{1}{1+\ep}} d_{\fg}(x,x_0)-C(n,\ep)\le\rho(x)\le d_{\fg}(x,x_0)+C(n).
\end{equation}
Similarly to in \cite{CZ10} or \cite{Zhang11}, we define:
$$D(s):=\{x\in X,\rho(x)\le s\},\qquad V(s):=\int_{D(s)\cap\mathcal{R}_X}d\fg,\qquad\chi(s):=\int_{D(s)\cap\mathcal{R}_X}Rd\fg.$$
Then the following classical equations from \cite{CZ10} are verified in \cite{CMZ23}

\begin{Lemma}[{\cite[Lemma 8.4]{CMZ23}}]\label{lemVX}
$V(s)$, $\chi(s)$ are absolutely continuous and for almost $s>0$  we have
\begin{equation}\label{ineq5}
    \tfrac{n}{2}V(s)-\tfrac{s}{2}V'(s)=\chi(s)-\tfrac{2}{s}\chi'(s).
\end{equation}
\begin{equation}\label{ineq6}
    \chi(s)\le\tfrac{n}{2}V(s).
\end{equation}
\end{Lemma}

\subsubsection{Scalar curvature lower bound}

We also need a scalar curvature lower bound for noncollapsed $\mathbb{F}$-limit metric solitons which are non-Ricci-flat. This result derived in \cite{CMZ23} is analogous to the classical case \cite{CLY11}.

\begin{Theorem}\label{Thm:preliminary scalar lower bound}
    Let $(\mathcal{R}_X,\fg,f_0)$ be the regular part of a noncollapsed $\mathbb{F}$-limit metric soliton and $W$ be the  soliton entropy. Suppose $\Ric$ does not vanish identically on $\mathcal{R}$. Then there is a constant $c$ depending on the soliton and a dimensional constant $C_0$, such that
    \begin{align*}
        R\ge \frac{c}{f_0-W+C_0}\quad \text{ on }\quad \mathcal{R}_X.
    \end{align*}
\end{Theorem}

\subsection{Size of singular set}

A singular point is distinguished from a regular point by the sizes of their curvature radii. In a  Ricci flow $(M^n,g_t)_{t\in I }$, the parabolic radius is defined as
    \begin{align*}
    r_{\Rm}(x,t):= \sup\{r>0\,:\,|\Rm|\le r^{-2}\ \text{ on }\ P(x,t;r)\},
\end{align*}
where
$$P(x,t;r):=B_{g_t}(x,r)\times \left([t-r^2,t+r^2]\cap I\right)$$
is the forward-and-backward parabolic neighborhood.

The effective strata is defined as:
\begin{Definition}[{\cite[Definition 11.1]{Bam20b}}]\label{Def:quantitative strata}
    Let $(M^n,g_t)_{t\in I }$ be a Ricci flow. Then for $\varepsilon>0$ and $0<r_1<r_2\le \infty$, define
    \begin{align*}
        \tilde{S}^{\varepsilon,0}_{r_1,r_2}\subset \tilde{S}^{\varepsilon,1}_{r_1,r_2}\subset \tilde{S}^{\varepsilon,2}_{r_1,r_2} \subset \hdots\subset \tilde{S}^{\varepsilon,n+2}_{r_1,r_2}\subset M\times I.
    \end{align*}
    $(x',t')\in \tilde{S}^{\varepsilon,k}_{r_1,r_2}$ if none of the following two holds for any $r'\in(r_1,r_2)$
    \begin{enumerate}[(1)]
        \item $(x',t')$ is $(\varepsilon,r')$-selfsimilar and weakly $(k+1,\varepsilon,r')$-split;
        \item $(x',t')$ is $(\varepsilon,r')$-selfsimilar, $(\varepsilon,r')$-static, and weakly $(k-1,\varepsilon,r')$-split.
    \end{enumerate}
\end{Definition}

The definition of selfsimilar, static, and split points shall be omitted. The point we want to make is that, for a noncollapsed Ricci flow, the complement of $\tilde{S}^{\varepsilon,n-2}_{r_1,r_2}$ consists only of regular points.

\begin{Proposition}[Bamler's partial regularity, {\cite[Corollary 17.2]{Bam20b}}]\label{Prop:partial regularity}
    For $Y<+\infty$, there is an $\varepsilon(Y)>0$ with the following property. Suppose $(M^n,g_t)_{t\in I}$ is a Ricci flow with bounded curvature within each compact interval. Assume for some $r>0$ and $(x,t)\in M\times I$ that $[t-\varepsilon^{-1}r^2,t]\subset I$ and $\mathcal{N}_{x,t}(r^2)\ge -Y$. Suppose either one of the following is true
    \begin{enumerate}[(1)]
        \item $(x,t)$ is weakly $(n-1,\varepsilon,r)$-split;
        \item $(x,t)$ is $(\varepsilon,r)$-static and weakly $(n-3,\varepsilon,r)$-split.
    \end{enumerate}
    Then $r_{\Rm}(x,t)\ge \varepsilon r$.
\end{Proposition}

We remark that in Bamler's original statement (see \cite[Assumption 15.6]{Bam20b}), the condition of the last proposition is ``strongly split'' instead of ``weakly split''. However, these two notions are essentially equivalent in view of \cite[\S 12]{Bam20b}.

The sizes of singular strata can be estimated as below.

\begin{Proposition}[Size of strata, {\cite[Proposition 11.2]{Bam20b}}]\label{Prop:size of strata}
   If $Y<\infty$, $\varepsilon>0$, then the following holds. Let $(M,g_t)_{t\in I}$ be a Ricci flow with bounded curvature within each compact time interval. For $(x_0,t_0)\in M\times I$ and $r>0$ assume $[t_0-2r^2,t_0]\subset I$ and $\mathcal{N}_{x_0,t_0}(r^2)\ge -Y$. Then for any $\sigma\in(0,\varepsilon)$, there are points $(x_i,t_i)\in \tilde{S}^{\varepsilon,k}_{\sigma r,\varepsilon r}\cap P^*(x_0,t_0; r),\ i=1,2,\hdots, N$ with $N\le C(Y,\varepsilon)\sigma^{-k-\varepsilon}$ and
   \begin{align*}
       \tilde{S}^{\varepsilon,k}_{\sigma r,\varepsilon r}\cap P^*(x_0,t_0; r)\subset \bigcup_{i=1}^NP^*(x_i,t_i; \sigma r).
   \end{align*}
\end{Proposition}

Recall the $P^*$-parabolic neighborhood is defined as
\begin{align*}
    P^*(x_0,t_0;A,T^+,-T^-):= &\ \{(x,t)\,:\, t\in[t_0-T^-,t_0+T^+],\ d_{W_1}^{g_{t_0-T^-}}(\nu_{x_0,t_0\,;\,t_0-T^-},\nu_{x,t\,;\,t_0-T^-})\le A\},
    \\
    P^*(x_0,t_0;r):= &\ P^*(x_0,t_0;r,r^2,-r^2).
\end{align*}

%-------------------------------Section3---------------------------------------------------------------------------------

\section{Curvature and volume estimates of metric solitons}

\subsection{A scalar curvature lower bound}

In this subsection, we prove Theorem \ref{thmR>}. Since our method is not completely the same as \cite[Proposition 1.1]{Ni05}, we summarize the idea as follows. Let us fix a large ball $B_{\fg}(x_0,D)$. By Theorem \ref{Thm:preliminary scalar lower bound}, since the soliton is non-Ricci-flat, there is a lower bound $\delta_0>0$ of $R$ on $B_{\fg}(x_0,D)\cap\mathcal{R}_X$. It is known that $\nabla f_0$ is complete, and that along the integral curve of $\nabla f_0$, the scalar curvature is always increasing due to the Ricci nonnegative assumption. Now, for any $x \in \mathcal{R}_X\setminus B_{\fg}(x_0,D)$, we consider the integral curve of $-\nabla f_0$ starting from $x$, there are two possibilities:
\begin{itemize}
    \item This integral curve intersects $ \mathcal{R}_X\setminus B_{\fg}(x_0,D)$ at some time. This then means that $R(x)$ is larger than $\delta_0$;
    \item This integral curve never intersect $ \mathcal{R}_X\setminus B_{\fg}(x_0,D)$. This means the integral curve approaches a critical point of $\nabla f_0$. But \eqref{eq: basic-properties-2} shows that a critical point in  $ \mathcal{R}_X\setminus B_{\fg}(x_0,D)$ is where the scalar curvature is large, so $R(x)$ would be even larger.
\end{itemize}

\begin{proof}[Proof of Theorem \ref{thmR>}]

Let $x_0$ be a center of the soliton and let $D>0$ be a constant to be determined. By Theorem  \ref{thm<f<} and Theorem \ref{Thm:preliminary scalar lower bound}, since the soliton is non-Ricci-flat, we can find a constant $\delta_0=\delta_0(D)>0$ depending on the soliton also, such that
\begin{align*}
    R\ge \delta_0 \quad \text{ on }\quad B_{\fg}(x_0,D)\cap \mathcal{R}_X.
\end{align*}

By \eqref{eq:Ricci identity and f}, we have
$$\langle\nabla f_0,\nabla R\rangle=2\Ric(\nabla f_0,\nabla f_0)\ge0,$$
hence along the integral curve of $\nabla f_0,$   $R$ is always increasing. Note that $\nabla f_0$ is complete by Theorem \ref{thmbase}(a). When applying Theorem \ref{thm<f<}, we shall fix $\varepsilon=1,$ so that
\begin{align}\label{eq:f growth ks}
    \tfrac{1}{5}d^2_{\fg}(x_0,x)-C(n)\le f_0(x)-W\le\tfrac{1}{4}(d_{\fg}(x_0,x)+C(n))^2.
\end{align}
Henceforth the estimation constant $C$ may vary from line to line, so long as it is a dimensional constant.

Fix an $x\in \mathcal{R}_X\backslash B_{\fg}(x_0,D).$ If $R(x)>1,$ then there is nothing to prove, since we have already found a lower bound $1$ for $R(x)$. So we assume $R(x)\le1$. Let $\sigma(s)$ be an integral curve of $\nabla f_0$ starting from $x$. Precisely, we define:
\begin{align*}
    \sigma'(s) & = \nabla f_0\quad \text{ for all }s\le 0,
    \\
    \sigma(0) & = x.
\end{align*}
Now it is clear that 
$$R(\sigma(s))\le R(\sigma(0))=R(x)\le1 \quad \text{ for all }\quad s\le 0.$$

We claim that there exists an $s_1<0,$ such that $d_{\fg}(\sigma(s_1),x_0)\le D.$ Assume by contradiction that 
$$d_{\fg}(\sigma(s),x_0)>D\quad \text{ for all }\quad s\le 0.$$ 
Then by \eqref{eq: basic-properties-2} and \eqref{eq:f growth ks}, we have 
$$|\nabla f_0|^2(\sigma(s))=f_0(\sigma(s))-W-R(\sigma(s))\ge\tfrac{1}{5}d_{\fg}^2(\sigma(s),x_0)-C>\tfrac{1}{5}D^2-C.$$
We take $D$ to be large enough, such that $\tfrac{1}{5}D^2-C>0$. On the other hand, we estimate 
$$f_0(x)-f_0(\sigma(s))=\int_s^0\tfrac{\partial}{\partial\eta}f_0(\sigma(\eta))d\eta=\int_s^0|\nabla f_0|^2(\sigma(\eta))d\eta\ge|s|(\tfrac{1}{5}D^2-C).$$
By \eqref{eq:f growth ks} again we have
\begin{align*}
    |s|(\tfrac{1}{5}D^2-C)&\le\tfrac{1}{4}(d_{\fg}(x_0,x)+C)^2-\tfrac{1}{5}d_{\fg}^2(x_0,\sigma(s))-C\\
    &\le\tfrac{1}{4}(d_{\fg}(x_0,x)+C)^2-C.
\end{align*}
Since $x$ is fixed, we obtain a contradiction when $s$ is large enough. So there is a $s_1<0$ such that $\sigma(s_1)\in B_{\fg}(x_0,D)\cap\mathcal{R}_X$. Hence
$$R(x)\ge R(\sigma(s_1))\ge\inf_{y\in B_{\fg}(x_0,D)\cap\mathcal{R}_X}R(y)\ge \delta_0.$$
So $\delta:=\min\{\delta_0,1\}$ is the desired lower bound for $R$.
\end{proof}

\subsection{A volume growth bound}

Next, we prove Theorem \ref{thmavr}. Given the validity of Lemma \ref{lemVX}, we recall that
    $$D(s):=\{x\in X,\rho(x)\le s\},\qquad V(s):=\int_{D(s)\cap\mathcal{R_X}}d\fg,\qquad \chi(s):=\int_{D(s)\cap\mathcal{R_X}}Rd\fg,$$
    where $\rho:=2\sqrt{f_0-W}\ge 0$. Define
  \begin{align*}
      P(s):= &\ \frac{V(s)}{s^n}-4\frac{\chi(s)}{s^{n+2}},
      \\
      Q(s):= &\ \frac{\chi(s)}{s^2V(s)}.
  \end{align*}
\begin{Lemma}\label{lem PQ}
    For almost every $s>0$, we have
    \begin{equation}\label{ineq8}
        \frac{V(s)}{s^n}\left(1-\frac{2n}{s^2}\right)\le P(s)\le\frac{V(s)}{s^n}.
    \end{equation}
    \begin{equation}\label{eq PP'}
        P'(s)=-s^{-1}(2s^2-4n-8)\frac{Q(s)}{1-4Q(s)}P(s),
    \end{equation}
\end{Lemma}
\begin{proof}
    By (\ref{ineq6}) we have
    \begin{equation}\label{ineq7}
        0\le Q(s)\le\frac{n}{2s^2}.
    \end{equation}
    Substituting it into the definition of $P(s)$, we have \eqref{ineq8}. 

By Lemma \ref{lemVX}, $P$ is also absolutely continuous, and almost everywhere satisfies:
    \begin{align*}
        P'(s)&=s^{-n} \left( V'(s)-nV(s)s^{-1}-4\chi'(s)s^{-2}+4(n+2)\chi(s)s^{-3}\right)\\
        &=s^{-n}\left(4s^{-2}\chi'(s)-2s^{-1}\chi(s)-4s^{-2}\chi'(s)+4(n+2)s^{-3}\chi(s)\right)\\
        &=-s^{-n-3}\chi(s)(2s^2-4n-8)\\
        &=-s^{-1}(2s^2-4n-8)\tfrac{Q(s)}{1-4Q(s)}P(s).
    \end{align*}
    And we get \eqref{eq PP'}.
\end{proof}
Now we prove Theorem \ref{thmavr}, and the proof is not different from \cite{Zhang11}.
\begin{proof}[Proof of Theorem \ref{thmavr}]
    Set $\zeta(s):=s^{-1}(2s^2-4n-8)\tfrac{Q(s)}{1-4Q(s)}$. If $s\ge s_0,$ where $s_0$ is some dimensional constant, then by (\ref{ineq8}) and (\ref{ineq7}) we obtain $\zeta(s)>0$ and $P(s)\ge\tfrac{V(s)}{2s^n}\ge0$, and thus $P'(s)=-\zeta(s)P(s)\le0.$
    Integrating from $s_0$ to $s$,
    \begin{align*}
        P(s)&=P(s_0)\exp\left(-\int_{s_0}^s\zeta(t)dt\right)\\
        &\le P(s_0)\exp\left(-\int_{s_0}^st^{-1}(2t^2-4n-8)\tfrac{\delta}{t^2}dt\right)\\
        &=P(s_0)\exp\left(\delta(2n+4)(s_0^{-2}-s^{-2})\right)s_0^{2\delta}s^{-2\delta}\\
        &\le P(s_0)\exp\left(\delta(2n+4)s_0^{-2}\right)s_0^{2\delta}s^{-2\delta}\\
        &= P(s_0)C(n,\delta)s^{-2\delta},
    \end{align*}
    where the inequality in the second line above is due to $$Q(t)=\frac{\chi(t)}{t^2V(t)}=t^{-2}\fint_{\{\rho\le t\}\cap\mathcal{R}_X}Rd\fg\ge\frac{\delta}{t^2}.$$ Then by (\ref{ineq4}) we have $B_{\fg}(x_0,\tfrac{3}{2}(s+c_n))\cap\cR_X\subset\{\rho\le s\}.$ Note that $P(s)\ge\tfrac{V(s)}{2s^n}$ for $s\ge s_0$, and hence
    $$\left|B_{\fg}(x_0,\tfrac{3}{2}(s+c_n))\cap\cR_X \right|\le V(s)\le2P(s)s^n\le2P(s_0)C(n,\delta)s^{n-2\delta},$$
    which means that for all $A$ large enough,
    $$|B_{\fg}(x_0,A)\cap\cR_X|\le C(n,\delta)P(s_0)s_0^{2\delta}\left(\tfrac{2}{3}\right)^{n-2\delta}A^{n-2\delta}\le CA^{n-2\delta}.$$The proof is done.
\end{proof}

\subsection{AVR of metric solitons}

To prove Theorem \ref{thm existence of AVR}, we need the following lemma.
\begin{Lemma}\label{lem limlim}
   Let $(X,d,\nu)$ be (the model of) a noncollapsed $\mathbb{F}$-limit metric soliton, $(\mathcal{R}_X,\fg,f_0)$ be its regular part, and $x_0$ be a center of the metric soliton.  Then the limit 
    $\lim_{A\to+\infty}\frac{|B_{\fg}(x_0,A)\cap\cR_{X}|}{ A^n}$ exists if and only if the limit $\lim_{s\to+\infty}\frac{V(s)}{s^n}$ exists, in which case 
    \begin{equation*}
        \lim_{A\to+\infty}\frac{|B_{\fg}(x_0,A)\cap\cR_{X}|}{ A^n}=\lim_{s\to+\infty}\frac{V(s)}{s^n}.
    \end{equation*}
\end{Lemma}
\begin{proof}
    By the definition of $V(s)$ and \eqref{ineq4}, for any $\varepsilon>0$ we have,
    \begin{equation*}
        |B_{\fg}(x_0,s-C)\cap\cR_X|\le V(s)\le \left|B_{\fg} \left( x_0,\sqrt{1+\ep}(s+C(\ep))\right) \cap\cR_X\right|\quad \text{ for all }\ s>0,
    \end{equation*}
    where we omit the dependence on $n$ in the constants above. Thus, it is clear that
    \begin{align*}
        \liminf_{s\to+\infty}\frac{|B_{\fg}(x_0,s-C)\cap\cR_X|}{s^n} &\ \le \liminf_{s\to+\infty}\frac{V(s)}{s^n}\le \limsup_{s\to+\infty}\frac{V(s)}{s^n}
        \\
        &\ \le\limsup_{s\to\infty}\frac{\left|B_{\fg} \left( x_0,\sqrt{1+\ep}(s+C(\ep))\right) \cap\cR_X\right|}{s^n}.
    \end{align*}
    This immediately implies that
    \begin{align*}
        \liminf_{s\to+\infty}\frac{|B_{\fg}(x_0,s)\cap\cR_{X}|}{s^n} &\ \le \liminf_{s\to+\infty}\frac{V(s)}{s^n}\le \limsup_{s\to+\infty}\frac{V(s)}{s^n} 
        \\
        &\ \le (1+\varepsilon)^{\frac{n}{2}}\limsup_{s\to+\infty}\frac{|B_{\fg}(x_0,s)\cap\cR_{X}|}{s^n}.
    \end{align*}
    In like manner, we also have
    \begin{align*}
        \liminf_{s\to+\infty}\frac{V(s)}{s^n}\le &\  (1+\varepsilon)^{-n/2} \liminf_{s\to+\infty}\frac{|B_{\fg}(x_0,s)\cap\cR_{X}|}{s^n}
        \\
        \le &\ \limsup_{s\to+\infty}\frac{|B_{\fg}(x_0,s)\cap\cR_{X}|}{s^n}
        \le  \limsup_{s\to+\infty}\frac{V(s)}{s^n}.
    \end{align*}
Since $\varepsilon>0$ is arbitrary, taking $\varepsilon\to0$, the proof is completed.
\end{proof}

Theorem \ref{thm existence of AVR} now follows from the corollary below.

\begin{Corollary}\label{coro PAVR}
    For a metric soliton described above, the $\AVR=\lim_{A\to+\infty}\frac{|B_{\fg}(x_0,A)\cap\cR_{X}|}{\omega_n A^n}$ exists and 
    $$\lim_{s\to+\infty}P(s) =\omega_n\AVR \in [0,C(n)e^{W}],$$ 
    where $W\le 0$ is the soliton entropy. In particular, the $\AVR$ is bounded from above by a constant depending only on $n$.
\end{Corollary}
\begin{proof}
    As we have seen in the proof of Theorem \ref{thmavr}, $P'(s)\le0$ for all $s\ge s_0$ large enough, which means $\lim_{s\to+\infty}P(s)$ exists. Then by \eqref{ineq8} and Lemma \ref{lem limlim}, 
    $$\lim_{s\to+\infty}P(s)=\lim_{s\to+\infty}\tfrac{V(s)}{s^n} =\omega_n\AVR.$$
    The upper bound of the $\AVR$ is a direct consequence of \cite[Corollary 1.3]{CMZ23}.
\end{proof}

 Now we show a characterization of when the $\AVR$ is positive; see also \cite{CLY12}.
\begin{Corollary}
    For a metric soliton described above we have that
    $$\AVR>0\; \text{if and only if there exists}\;\; s_0>0,\;\text{such that} \int_{s_0}^{+\infty}\tfrac{\chi(t)}{tV(t)}dt<+\infty$$
\end{Corollary}
\begin{proof}
    As in the proof of Theorem \ref{thmavr}, for $s\ge s_0$, where $s_0$ is a dimensional constant,  we have
    \begin{equation}\label{eq PsPs_0}
        P(s)=P(s_0)\exp\left(-\int_{s_0}^st^{-1}(2t^2-4n-8)\tfrac{Q(t)}{1-4Q(t)}dt\right).
    \end{equation}
    By \eqref{ineq7}, for all $s\ge s_0$,
    $$\tfrac{3}{2}\int_{s_0}^stQ(t)dt\le \int_{s_0}^s t^{-1}(2t^2-4n-8)\tfrac{Q(t)}{1-4Q(t)}dt \le 2\int_{s_0}^stQ(t)dt.$$
    So \eqref{eq PsPs_0} becomes 
    $$P(s_0)e^{-2\int_{s_0}^s \tfrac{\chi(t)}{tV(t)}dt}\le P(s)\le P(s_0)e^{-\frac{3}{2}\int_{s_0}^s \tfrac{\chi(t)}{tV(t)}dt}.$$
    Then by Corollary \ref{coro PAVR} we have finished the proof. 
\end{proof}

%---------------------------------------Section 4 ---------------------------------------

\section{Ancient Ricci flow with nonnegative Ricci curvature}

In this section and the next, we shall consider noncollapsed ancient solutions. Since one side of the dichotomies in Theorem \ref{thm tangent flow character1} and Theorem \ref{thm tangent flow character2} is that every tangent flow at infinity is a Ricci flat cone, we need only to focus on the case where \emph{one} tangent flow at infinity is non-Ricci-flat in view of Theorem \ref{thmbase}(e). Let us recall our setting.

Let $(M^n,g_t)_{t\in(-\infty,0]}$ be an ancient Ricci flow with bounded curvature within each compact time interval. Furthermore, suppose for some $y_0\in M$ and $Y<+\infty$, we have
\begin{align}\label{eq:noncollapsing of ancient solution}
    \mathcal{N}_{y_0,0}(\tau)\ge -Y\quad \text{ for all }\quad \tau>0.
\end{align}
Let $\lambda_i\nearrow+\infty$ be a sequence of scales, and define
\begin{align*}
    g^i_t:=\lambda_i^{-1}g_{\lambda_it}.
\end{align*}
Then, there is a metric soliton $(\mathcal{X},(\nu_t)_{t\in(-\infty,0)})$, called a \emph{tangent flow at infinity}, and a correspondence $\mathfrak{C}$, such that, after passing to a subsequence,
\begin{align}\label{def-limiting-condition}
    \left((M,g^i_t)_{t\in(-\infty,0]},(\nu^i_{y_0,0;t})_{t\in(-\infty,0]}\right)\xrightarrow[i\to\infty]{\mathbb{F},\mathfrak{C}}\big(\mathcal{X},(\nu_t)_{t\in(-\infty,0)}\big).
\end{align}

Since the limit flow is a metric soliton, we remark the following important fact.

\begin{Lemma}\label{Lm:timewise}
    After passing to a further subsequence, the convergence \eqref{def-limiting-condition} is time-wise at any $t\in(-\infty,0)$.
\end{Lemma}
\begin{proof}
In view of the self-similarity of the metric soliton, the function
\begin{align*}
    t\mapsto \iint_{\mathcal{X}_t\times \mathcal{X}_t} d_t(x,y)\,d\nu_t(x)d\nu_t(y)
\end{align*}
varies only by scaling, and hence is continuous. Then by \cite[Definition 4.7, Theorem 4.9]{Bam23}, the metric soliton  $\big(\mathcal{X},(\nu_t)_{t\in(-\infty,0)}\big)$ is continuous at any $t\in(-\infty,0)$. Since the limit $\big(\mathcal{X},(\nu_t)_{t\in(-\infty,0)}\big)$ is always continuous, it must be the unique representative of the limit in the statement of \cite[Theorem 7.6]{Bam23}, for any other equivalent representative shall not be future continuous. Thus, by \cite[Theorem 7.6]{Bam23}, after passing to a subsequence, \eqref{def-limiting-condition} is time-wise at any continuous time of $\big(\mathcal{X},(\nu_t)_{t\in(-\infty,0)}\big)$, namely, at any $t\in(-\infty,0)$.
    
\end{proof}

Note that the soliton entropy of $(\mathcal{X},(\nu_t)_{t\in(-\infty,0)})$ is equal to
\begin{align*}
    W:=\lim_{\tau\to\infty}\mathcal{N}_{y_0,0}(\tau).
\end{align*}

Theorem \ref{thm tangent flow character1} is  reduced to the following proposition.

\begin{Proposition}\label{Prop: zero avr}
    Let  $(M^n,g_t)_{t\in(-\infty,0]}$ be an ancient solution as described above, with bounded curvature within each compact time interval, and satisfying the noncollapsing condition \eqref{eq:noncollapsing of ancient solution}. Assume furthermore that $\Ric\ge 0$ on $M\times(-\infty,0]$. If a tangent flow at infinity obtained in \eqref{def-limiting-condition} has non-Ricci-flat regular part, then
    \begin{align*}
        \AVR_{g_t}\equiv 0\quad \text{ for all }\quad t\le 0.
    \end{align*}
\end{Proposition}

We briefly summarize the ideas of the proof of Proposition \ref{Prop: zero avr}. First of all, we recall the fact that the $\AVR$ is non-increasing for the ancient flow in question, so it remains to show that
\begin{align*}
    \AVR_{g_{-\lambda_i}}\to0\quad \text{ or }\quad \AVR_{g^i_{-1}}\to 0.
\end{align*}
On the other hand, by Theorem \ref{thmavr}, the smooth part of the limit has zero $\AVR$. So it is sufficient to show that the volume of a large ball in $(M,g^i_{-1})$ converges in some sense to the volume of a large ball on the regular part of the limit.

\subsection{Monotonicity of the AVR}

The following lemma is a simple consequence of Perelman's distance distortion estimate.

\begin{Lemma}\label{Lm:Monotonicity of AVR}
Let $(M,g_t)_{t\in (-\infty,0])}$ be the ancient solution in Proposition \ref{Prop: zero avr}.  Then we have
\begin{align*}
    \AVR_{g_{t_1}} \ge \AVR_{g_{t_2}}\quad \text{ for } \ \ t_1\le t_2\le0.
\end{align*}
\end{Lemma}

\begin{proof}
    Fix $t_1<t_2$. There is  a positive number $K$ such that
    \begin{align*}
        |\Ric_{g_t}|\le K\quad \text{ on }\quad M\times[t_1,t_2].
    \end{align*}
    Let $x\in M$ be a fixed point. By Perelman's distance distortion estimate \cite[Lemma 8.3(b)]{Per02}, for any $A\gg 1$, we have
    \begin{align*}
        B_{g_{t_2}}(x,A)\subset B_{g_{t_1}}\Big(x,A+C(K)(t_2-t_1)\Big). 
    \end{align*}
    On the other hand, since $R\ge 0$, the volume form is decreasing. Thus
    \begin{align*}
        |B_{g_{t_2}}(x,A)|_{g_{t_2}} \le \left|B_{g_{t_1}}\Big(x,A+C(K)(t_2-t_1)\Big)\right|_{g_{t_1}}\quad \text{ for all }\ A\gg1.
    \end{align*}
    It is clear that
    \begin{align*}
        \AVR_{g_{t_2}}\le \AVR_{g_{t_1}}.
    \end{align*}
\end{proof}

\subsection{The small volume of the singular part}

We continue with the proof of Proposition \ref{Prop: zero avr}. Let $\{(M,g^i_t)_{t\in(-\infty,0]}\}_{i=1}^\infty$ and $(\mathcal{X},(\nu_t)_{t\in(-\infty,0)})$ be the sequence and the limit in  \eqref{def-limiting-condition} as described at the beginning of the section. $(y_0,0)$ is the fixed base point for each $(M,g^i_t)_{t\in(-\infty,0]}$. We shall write the regular part of $(\mathcal{X},(\nu_t)_{t\in(-\infty,0)})$ as $(\mathcal{R},\fg_t,f)_{t\in(-\infty,0)}$ with
\begin{align*}
    d\nu_t=(4\pi|t|)^{-n/2}e^{-f}d\fg_t.
\end{align*}
Let $(X,d,\nu)$ be the model of the metric soliton and $(\mathcal{R}_X,\fg,f_0)$ its regular part.

    Let $z_i\in M$ be a sequence of points such that for each $i$, $(z_i,-1)$ is an $H_n$-center of $(y_0,0)$ with respect to the scaled flow $g^i_t$ (or, in other words, $(z_i,-\lambda_i)$ is an $H_n$-center of $(y_0,0)$ with respect to the unscaled flow $g_t$). Our goal is to estimate
    \begin{align*}
        \left|B_{g^i_{-1}}(z_i,A)\right|_{g^i_{-1}},\quad A\gg 1,
    \end{align*}
    when $i$ is large enough. To this end, we decompose this geodesic ball into two parts:
    \begin{align*}
        \cR^{(i)}(r,A):=B_{g^i_{-1}}(z_i,A)\cap\{r_{\Rm}\ge r\},\quad \cS^{(i)}(r,A):=B_{g^i_{-1}}(z_i,A)\cap\{r_{\Rm}<r\},\quad r>0,\ \ A<+\infty.
    \end{align*}
    Obviously, 
    \begin{align*}
        B_{g^i_{-1}}(z_i,A) =\cR^{(i)}(r,A) \sqcup \cS^{(i)}(r,A).
    \end{align*}
    In this subsection and the next, we shall estimate  the volumes of $\cS^{(i)}(r,A)$ and  $\cR^{(i)}(r,A)$, respectively.

    The singular part $\cS^{(i)}(r,A)$ is estimated by Bamler's theory.

    \begin{Lemma}\label{Lm: volume of the singular part}
        If 
        \begin{align*}
            \quad A\ge 1,\quad r\le\overline{r}(A,Y),
        \end{align*}
        then there is an $\alpha=\alpha(Y)>0$, such that
        \begin{align*}
            \left|\cS^{(i)}(r,A)\right|_{g^i_{-1}}\le C(A,Y)r^{\alpha}\quad \text{ for all }\quad i.
        \end{align*}
        Here $Y$ is the entropy bound in \eqref{eq:noncollapsing of ancient solution}.
    \end{Lemma}

\begin{proof}
  By Proposition \ref{Prop:partial regularity}, if $\ep= \overline{\ep}(Y)\in(0,1)$ and $r<\varepsilon$, then for any $r'\ge r/\varepsilon$, the points in $\cS^{(i)}(r,A)$ are neither weakly $(n-1,\ep,r')$-split, nor $(\ep,r')$-static and weakly $(n-3,\ep,r')$-split at the same time. Thus, 
  \begin{align*}
      \cS^{(i)}(r,A) \subset \tilde{\mathcal{S}}_{r/\varepsilon,1}^{\varepsilon,n-2}(g^i),
  \end{align*}
  where $\tilde{\mathcal{S}}_{r/\varepsilon,1}^{\varepsilon,n-2}(g^i)$ is the quantitative strata with respect to the flow $g^i_t$ in Definition \ref{Def:quantitative strata}. 
  
  By the nonnegative Ricci curvature assumption and the Bishop-Gromov comparison theorem, we can find points $\{y_j^{(i)}\}_{j=1}^{N^{(i)}}\subset \cS^{(i)}(r,A)$ with
  \begin{align}\label{strata1}
      \cS^{(i)}(r,A)\subset \bigcup_{j=1}^{N^{(i)}} B_{g^i_{-1}}\left(y_j^{(i)},1/\varepsilon\right)\quad \text{ and } \quad N^{(i)}\le \overline{N}(A,\varepsilon).
  \end{align}
  The monotonicity of the $W_1$-Wasserstein distance \cite[Lemma 2.7]{Bam20a} implies that $B_{g^i_{-1}}\left(y_j^{(i)},1/\varepsilon\right)\subset P^*\left(y_j^{(i)},-1;1/\varepsilon\right)$, and thus
  \begin{align}\label{strata2}
      \cS^{(i)}(r,A)\subset \bigcup_{j=1}^{N^{(i)}} \left(\tilde{\mathcal{S}}_{r/\varepsilon,1}^{\varepsilon,n-2}(g^i)\cap P^*\left(y_j^{(i)},-1;1/\varepsilon\right)\right).
  \end{align}

Now, applying  Proposition \ref{Prop:size of strata} to $\tilde{\mathcal{S}}_{r/\varepsilon,1}^{\varepsilon,n-2}(g^i)\cap P^*\left(y_j^{(i)},-1;1/\varepsilon\right)$ with $\varepsilon\to\varepsilon$, $r\to 1/\varepsilon$, $\sigma \to r$, we can find $\{y^{(i)}_{j,k}\}_{k=1}^{N^{(i)}_j}$, such that
\begin{align}\label{strata3}
    \tilde{\mathcal{S}}_{r/\varepsilon,1}^{\varepsilon,n-2}(g^i)\cap P^*\left(y_j^{(i)},-1;1/\varepsilon\right) \subset \bigcup_{k=1}^{N^{(i)}_j} P^*\left(y_{j,k}^{(i)},-1;r/\varepsilon\right),\quad N^{(i)}_j \le C(Y,\varepsilon) r^{-n+2-\varepsilon}.
\end{align}

Since, by \cite[Theorem 9.8]{Bam20a}, 
\begin{align}\label{strata4}
    \left|P^*\left(y_{j,k}^{(i)},-1;r/\varepsilon\right)\bigcap M\times\{-1\}\right|_{g^i_{-1}}\le C(Y)r^n\varepsilon^{-n},\quad k= 1,2,\hdots,N^{(i)}_j,\ \ j=1,2,\hdots, N^{(i)},
\end{align}
we have, combining \eqref{strata1}\eqref{strata2}\eqref{strata3}\eqref{strata4},
\begin{align*}
    \left|\cS^{(i)}(r,A)\right|_{g^i_{-1}}\le C(Y,A,\varepsilon)r^{2-\varepsilon}.
\end{align*}
Since $\varepsilon=\varepsilon(Y)$, we obtain the lemma by letting $\alpha=2-\varepsilon>0$.

\end{proof}

\subsection{The volume estimate of the regular part}

In this subsection, we proceed to estimate the size of $\mathcal{R}^{(i)}(r,A)$ via the local smooth convergence on $\mathcal{R}$.

\begin{Lemma}\label{Lm:smooth volume convergence}
    If $A\ge\underline{A}(n)$, $r\in(0,1)$, then passing to a subsequence,  we have
    \begin{align*}
        \left|\cR^{(i)}(r,A)\right|_{g^i_{-1}}\le 2\left|B_{\fg}(x_0,2A)\cap\mathcal{R}_X\right|\quad \text{ for all }\ i\ \text{ large enough}.
    \end{align*}
\end{Lemma}

\begin{proof}

Let 
\begin{align}\label{eq:limiting diffeo}
    \psi_i:\mathcal{R}\supset U_i\to V_i\subset M\times (-\infty,0]
\end{align}
be the sets and diffeomorphism defined in Theorem \ref{Thm_smooth_convergence}. By the smooth convergence on $\mathcal{R}$, it is obviously sufficient to show that
\begin{align*}%\label{eq:covering the regular part}
    M\times\{-1\}\supset \psi_i\left(U_i\cap B_{\fg}(x_0,2A)\right)\supset \cR^{(i)}(r,A)
\end{align*}
for all $i$ large enough, depending on $r$ and $A$.

    For each $i$, we find a finite sequence of points $\left\{x^{(i)}_k\right\}_{k=1}^{N_i}$ in $\cR^{(i)}(r,A)$ such that 
    \begin{align*}
        \cR^{(i)}(r,A)\subset \bigcup_{k=1}^{N_i} B_{g^i_{-1}}\left(x_k^{(i)},\tfrac{r}{3}\right),\qquad N_i\le \overline{N}(A,r),
    \end{align*}
    by the Bishop-Gromov comparison theorem, where $\overline{N}(A,r)$ is independent of $i$. By passing to a (not relabeled) subsequence, we may assume that $N_i\equiv N \le \overline{N}(A,r)$ for all $i$. Next, we show that, for each $k\in\{1,2,\hdots,N\}$, the sequence $\left\{x_k^{(i)}\right\}_{i=1}^{\infty}$ converges, after passing to a subsequence, to some point in $\cR_{-1}$. 
    \\

%---------------------claim-----------begin---------------------
\noindent\textbf{Claim. } \emph{
For each $k\in\{1,2,\hdots,N\}$, there is a point $\tilde{x}_k\in B_{\fg}(x_0,A+C(n))\subset\mathcal{R}_X=\mathcal{R}_{-1}$, such that, after passing to a subsequence,}   
\begin{align}\label{eq:points converging smooth}
    \psi_i^{-1}\left(x^{(i)}_k,-1\right)\xrightarrow{i\to\infty} \tilde{x}_k.
\end{align}
\emph{Furthermore, $P\left(\tilde x_k;\tfrac{1}{2}r\right)$ is unscathed and}
\begin{align}\label{eq:cube converging smooth}
   \psi_i\left( P\left(\tilde x_k;\tfrac{1}{2}r\right)\right)\supset P\left( x^{(i)}_k,-1;\tfrac{1}{3}r\right)\quad\text{for $i$ large enough}.
\end{align}

\begin{proof}[Proof of the claim]
    For the sake of notational simplicity, we shall denote $x^{(i)}_k$ by $x^{(i)}$ and $\tilde{x}_k$ by $\tilde{x}$, and suppress the subindex.  Let $0<r'\ll r$ be a small number to be fixed, and set $t'=-1-r'^2.$ Since $x^{(i)}\in B_{g^i_{-1}}(z_i,A)$, we have, by \cite[Lemma 3.2, Definition 3.10]{Bam20a}
    \begin{align*}
        d^{g^i_{-1}}_{W_1}\left(\delta_{x^{(i)}}, \nu^i_{y_0,0\,;\,-1}\right) \le &\ d^{g^i_{-1}}_{W_1}\left(\delta_{x^{(i)}}, \delta_{z_i}\right) + d^{g^i_{-1}}_{W_1}\left(\delta_{z_i}, \nu^i_{y_0,0\,;\,-1}\right)  \le A+\sqrt{2H_n}.
    \end{align*}
    Thus, by Theorem \ref{Thm:compactness of points}, we may find a conjugate heat flow $(\tilde{\nu}_t)_{t<-1}$ on the metric soliton $(\mathcal{X},\nu_t)$, such that
    \begin{equation*}
        \left(\nu^i_{x^{(i)},-1\,;\,t}\right)_{t\le -1}\xrightarrow[i\to\infty]{\fC} (\tilde{\nu}_t)_{t<-1},\qquad \lim_{t\nearrow -1}\Var(\tilde{\nu}_t)=0,
    \end{equation*}
   where $\mathfrak{C}$ is the correspondence in \eqref{def-limiting-condition}. Since the limit flow pair is also $H_n$-concentrated \cite[Theorem 7.4]{Bam23}, by \cite[Definition 2.6, Proposition 3.23]{Bam23}, we have
    \begin{align*}
        \iint_{\mathcal{X}_t\times\mathcal{X}_t}d_t^2(x,y)\,d\tilde\nu_t(x)d\tilde\nu_t(y) =\Var(\tilde\nu_t)\le H_n|t+1|\quad \text{ for all }\quad t<-1
    \end{align*}
    In particular, we can find a $y'\in\mathcal{X}_{t'}$, such that
    \begin{align*}
        \Var(\delta_{y'},\tilde{\nu}_{t'})=\int_{\mathcal{X}_{t'}}d_{t'}^2(x,y')\,d\tilde\nu_{t'}(x) \le \iint_{\mathcal{X}_{t'}\times\mathcal{X}_{t'}}d_{t'}^2(x,y)\,d\tilde\nu_{t'}(x)d\tilde\nu_{t'}(y) \le H_n|t'+1|.
    \end{align*}
    Since $\mathcal{R}_{t'}$ is dense in $\mathcal{X}_{t'}$, we may, without loss of generality, assume $y'\in \mathcal{R}_{t'}$, this causes at most an additional dimensional factor. We then have, by \cite[Lemma 2.8]{Bam23} and $t'=1-r'^2$,
    \begin{equation}\label{ineqHn12}
        d_{W_1}^{\mathcal{X}_{t'}}(\delta_{y'},\tilde\nu_{t'})\le\sqrt{\Var(\delta_{y'},\tilde\nu_{t'})}\le C(n)r'.
    \end{equation}
Since $y'\in\mathcal{R}_{t'}\subset \mathcal{R}$, for all $i$ large enough, we may set
    $$(y'_i,t'):=\psi_i(y')\in M\times\{t'\},$$
    according to Theorem \ref{Thm_smooth_convergence}, where $\psi_i$ is the diffeomorphism in \eqref{eq:limiting diffeo}.
    
    Next, we estimate the space-time distance between $\left(x^{(i)},-1\right)$ and $(y'_i,t').$ Since by Lemma \ref{Lm:timewise}, the convergence \eqref{def-limiting-condition} is time-wise at $t'$, we have (see \cite[Definition 6.7]{Bam23})
    \begin{equation}\label{ineqdW1Z1}
        d_{W_1}^{Z_{t'}}\left((\varphi_{t'}^i)_*\nu_{x^{(i)},-1;t'}^i,(\varphi_{t'})_*\tilde\nu_{t'}\right)\to0.
    \end{equation}
$\fC=(Z_t,\varphi^i_t,\varphi_t)$ is the correspondence, $\varphi_{t}^i:(M,g^i_t)\to (Z_t,d_t^Z)$ and $\varphi_t:(\cX_t,d_t)\to (Z_t,d^Z_t)$ are isometric embeddings.

    Putting (\ref{ineqHn12}) into $\fC$ yields
    \begin{equation}\label{ineqdW1Z2}
        d_{W_1}^{Z_{t'}}\left( (\varphi_{t'})_*\delta_{y'} ,(\varphi_{t'})_*\tilde\nu_{t'} \right)\le C(n)r'.
    \end{equation}
  The definition of $(y_i',t')$ implies that (see Theorem \ref{Thm_smooth_convergence}(c))
    \begin{equation}\label{ineqdW1Z3}
        d_{W_1}^{Z_{t'}}\left((\varphi_{t'}^i)_*\delta_{y_i'},(\varphi_{t'})_*\delta_{y'}\right)\to0.
    \end{equation}
    By (\ref{ineqdW1Z1})-(\ref{ineqdW1Z3}) and the triangle inequality, we have
        $d_{W_1}^{Z_{t'}}\left( (\varphi_{t'}^i)_*\nu_{x^{(i)},-1;t'}^i,  (\varphi_{t'}^i)_*\delta_{y_i'}  \right) \le C(n)r'$ for all $i$ large enough. This is then equivalent to
    \begin{equation}\label{ineq not spacetime}
        d_{W_1}^{g^i_{t'}}(\nu^i_{x^{(i)},-1;t'},\delta_{y_i'})\le C(n)r' \quad \text{ for all }\ i\ \text{ large enough.}
    \end{equation}

    By the assumption $x^{(i)}\in\cR^{(i)}(r,A)$,  we have $r_{\Rm}(x^{(i)},-1)\ge r$. While $r'\ll r$, we have  $|\Ric|\le r'^{-2}$ on $P(x^{(i)},-1;r')$. Thus, by \cite[Theorem 9.5]{Bam20a},
    $d_{W_1}^{g^i_{t'}}(\nu^i_{x^{(i)},-1;t'},\delta_{x^{(i)}})\le C(n)r'$. 
    In combination with (\ref{ineq not spacetime}), we have
    \begin{equation*}
       d_{g^i_{-1}}\left(x^{(i)},y_i'\right)\le d_{g^i_{t'}}\left(x^{(i)},y_i'\right)=d_{W_1}^{g^i_{t'}}(\delta_{x^{(i)}},\delta_{y_i'}) \le C(n)r',
    \end{equation*}
    where the first inequality is due to the nonnegative assumption of the Ricci curvature. Thus, we have $(y_i',t')\in P\left(x^{(i)},-1;C(n)r'\right).$ 

    Now we fix $r'=r'(r)\ll r$ to ensure that, via the standard distance distortion estimate, 
    \begin{align*}
        r_{\Rm}(y_i',t')\ge \tfrac{99r}{100}\quad \text{ and }\quad \left(x^{(i)},-1\right)\in P\left(y_i',t';\tfrac{1}{100}r\right).
    \end{align*}
    By Theorem \ref{Thm_smooth_convergence}, the parabolic cube $P\left(y';\tfrac{98r}{100}\right)\Subset\mathcal{R}$ is unscathed. Note that no world-line in the smooth part of a metric soliton can become extinct in time, since the world-lines are simply integral curves of the complete field $\nabla f_0$. By the smooth convergence on $P\left(y';\tfrac{98r}{100}\right)$, we then have
    \begin{align*}
        \psi_i^{-1}\left(x^{(i)},-1\right)\in P\left(y';\tfrac{98r}{100}\right)\quad \text{ for all }\ i\ \text{ large enough},
    \end{align*}
    and, after passing to a sequence, 
    \begin{align*}
        \psi^{-1}_i\left(x^{(i)},-1\right)\to \tilde{x}\in P\left(y';\tfrac{2r}{100}\right)\cap\mathcal{R}_{-1}.
    \end{align*}
    If we take $r'$ to be even smaller if necessary (depending only on $n$ and $r$), then the standard distance distortion estimate implies that
    $$P\left(\tilde x;\tfrac{1}{2}r\right)\subset P\left(y';\tfrac{98r}{100}\right),$$
    and \eqref{eq:points converging smooth}\eqref{eq:cube converging smooth} follow easily by the smooth convergence on $P\left(\tilde x;\tfrac{1}{2}r\right)\Subset \mathcal{R}_{-1}$.

To see $\tilde{x}\in B_{\fg}(x_0,A+C(n))$, we compute
\begin{align*}
    d_{\fg}(x_0,\tilde{x})\le &\ d^Z_{-1}\left(\varphi_{-1}(\tilde{x}),\varphi^i_{-1}\left(x^{(i)}\right)\right)+d_{g^i_{-1}}\left(x^{(i)},z_i\right) + d^Z_{-1}\left(\varphi_{-1}(x_0),\varphi^i_{-1}\left(z_i\right)\right)
    \\
    \le &\ d^{Z}_{-1}\left(\varphi_{-1}(\tilde{x}),\varphi^i_{-1}\left(x^{(i)}\right)\right)+d_{g^i_{-1}}\left(x^{(i)},z_i\right) 
    \\
    &\ + d^{Z_{-1}}_{W_1}\left((\varphi_{-1})_*(\nu_{-1}),(\varphi^i_{-1})_*\left(\nu^i_{y_0,0\,;\,-1}\right)\right) +d^{\fg}_{W_1}\left(\delta_{x_0},\nu_{-1}\right)+d^{g^i_{-1}}_{W_1}\left(\delta_{z_i},\nu^i_{y_0,0\,;\,-1}\right).
\end{align*}
Since the convergence \eqref{def-limiting-condition} is time-wise at $t=-1$ by Lemma \ref{Lm:timewise}, taking a limit at the right-hand-side of the above inequality, we conclude from Theorem \ref{Thm_smooth_convergence}(c) that
$$d_{\fg}(x_0,\tilde x)\le C(n)+A.$$

\end{proof}

%---------------------------------end----------------------------claim------------

We continue with the proof of the lemma. Applying the claim successively to $\{x^{(i)}_k\}_{i=1}^\infty$ for $k=1,2,\hdots,N$, we find $\tilde x_k\in B_{\fg}(x_0,A+C(n))\cap\mathcal{R}_X$, $k=1,2,\hdots,N$, such that after passing to subsequences for finitely many times, \eqref{eq:points converging smooth} and \eqref{eq:cube converging smooth} both hold. Thus, if $r<1$ and $A\ge \underline{A}(n)$, we have
\begin{align*}
       K:=\left( \bigcup_{k=1}^N P\left(\tilde{x}_k;\tfrac{1}{2}r\right)\right)\cap\mathcal{R}_{-1}\Subset B_{\fg}(x_0,A+1+C(n))\cap\cR_X\subset B_{\fg}(x_0,2A)\cap\cR_X,
\end{align*}
for all $i$ large enough. Since 
\begin{align*}
    \psi_i(K)\supset \left(\bigcup_{k=1}^NP\left(x^{(i)}_k,-1;\tfrac{1}{3}r\right)\right)\cap M\times\{-1\} = \bigcup_{k=1}^NB_{g^i_{-1}}\left(x^{(i)}_k,\tfrac{1}{3}r\right)\supset\mathcal{R}^{(i)}(r,A), 
\end{align*}
the conclusion follows from the smooth convergence on $K\Subset\mathcal{R}_{-1}$.  
\end{proof}

\subsection{Completion of the proof}

\begin{proof}[Proof of Proposition \ref{Prop: zero avr}]
    
    Since the tangent flow at infinity is non-Ricci-flat, we may assume that 
    \begin{align*}
        R_{\fg}\ge \delta>0
\quad \text{ on }  \quad   \mathcal{R}_X=\mathcal{R}_{-1}.
\end{align*}
Then, by Theorem \ref{thmavr}, Lemma \ref{Lm: volume of the singular part}, and Lemma \ref{Lm:smooth volume convergence}, if we fix $A\ge\underline{A}(n)$ and $r\le \overline{r}(A,Y)$ and pass to a subsequence, we have
\begin{align*}
    \AVR_{g^i_{-1}}\le \frac{\left|B_{g^i_{-1}}(z_i,A)\right|}{\omega_n A^n}\le C(n)\frac{(2A)^{n-2\delta}}{A^n}+\frac{C(A,Y)r^{\alpha}}{A^n}\le C(n)A^{-2\delta}+C(A,Y)r^{\alpha},
\end{align*}
for all $i$ large enough; here we have applied the monotonicity of the volume ratio. Due to Lemma \ref{Lm:Monotonicity of AVR}, we have
\begin{align*}
    \lim_{t\to-\infty}\AVR_{g_t}=\lim_{i\to\infty}\AVR_{g^i_{-1}} \le C(n)A^{-2\delta}+C(A,Y)r^{\alpha},
\end{align*}
for all $A\ge\underline{A}(n)$ and $r\le \overline{r}(A,Y)$. Taking $r\to 0$ first and then $A\to \infty$, we get
\begin{align*}
     \lim_{t\to-\infty}\AVR_{g_t}=0,
\end{align*}
and hence, by Lemma \ref{Lm:Monotonicity of AVR},
\begin{align*}
    \AVR_{g_t}\equiv 0\quad \text{ for all }\quad t\in(-\infty,0].
\end{align*}

\end{proof}

% --------------------Sction 5----------------------------------------------------------------------------------------------------------------

\section{Ancient Ricci flow with positively pinched Ricci curvature}

We now consider an ancient solution described at the beginning of \S 4. As we have already explained, Theorem \ref{thm tangent flow character2} can be reduced to the following proposition.

\begin{Proposition}\label{Prop:pinching compact}
    Let $(M^n,g_t)_{t\in(-\infty,0]}$ be an ancient solution as in the statement of Proposition \ref{Prop: zero avr}. Assume that there is some $\varepsilon>0$, such that
    \begin{align*}
        \Ric_{g_t}\ge \varepsilon R_t g_t\quad \text{ on }\quad M\times(-\infty,0].
    \end{align*}
    If a tangent flow at infinity obtained in \eqref{def-limiting-condition} has non-Ricci-flat regular part, then $M$ is a closed manifold.
\end{Proposition}

The key idea of this proof is to improve the lower bound of the scalar curvature by convolution of heat kernel. Note that $R$ is a super-solution of the heat equation, so if $\int_M R\,d\nu_{x,t\,;\,t'}\ge \delta$, we would obtain a lower bound of $R$ near $(x,t)$. This local lower bound, by a covering argument, can be extended to a global lower bound, which compels the manifold to be compact by the pinching condition of the Ricci curvature.

\subsection{Local lower bound for the scalar curvature}

We apply Bamler's gradient estimate for the heat equation to obtain a local lower bound for the scalar curvature. The idea of the following lemma was also applied in \cite[Proposition 5.2]{CMZ23}.

\begin{Lemma}\label{Lm:local bound for scalar curvature}
Let $(M^n,g_t)_{t\in[-r^2,0]}$ be a Ricci flow with bounded curvature. Suppose furthermore that $R_{g_t}\ge 0$ on $M\times[-r^2,0]$. If for some $x\in M$, we have
\begin{align*}
    \int_M R_{g_{-r^2}}\,d\nu_{x,0\,;\,-r^2}\ge \delta r^{-2}>0,
\end{align*}
then we have
\begin{align*}
    R_{g_{0}}\ge n\sqrt{2}\Phi\left(\tfrac{1}{\sqrt{2}}\Phi^{-1}(\delta/n)-A\right) r^{-2}\quad \text{ on }\quad B_{g_0}(x,Ar).
\end{align*}
where $$\Phi(s)=\int_{-\infty}^s(4\pi)^{-1/2}e^{-r^2/4}dr,$$ is defined in \cite[Section 4]{Bam20a}.
\end{Lemma}

\begin{proof}
    By parabolic scaling, we may assume $r=1$. Consider the solution to the heat equation:
    \begin{align*}
        u:M\times[-1,0]\to &\ \mathbb{R},
        \\
        (y,s) \mapsto &\ \int_MR_{g_{-1}}\,d\nu_{y,s\,;\,-1}.
    \end{align*}
Since $u(\cdot,-1)=R_{g_{-1}}$, and the scalar curvature is a super-solution of the heat equation
\begin{align*}
    (\partial_t-\Delta)R_{g_t}=2\left|\Ric_{g_t}\right|^2\ge 0,
\end{align*}
we have
\begin{align*}
    R_{g_s}\ge u(\cdot,s)\quad \text{ on }\quad M\times[-1,0].
\end{align*}

By \cite[Proposition 5.13]{Bam20a},
\begin{align*}
    u(y,s)\in(0,n)\quad \text{ for all }\ (y,s)\in M\times[-1/2,0].
\end{align*}
Applying \cite[Theorem 4.1]{Bam20a} to $\frac{1}{n}u(\cdot,0)$, we have that $\frac{1}{\sqrt{2}}\Phi^{-1}\left(\frac{1}{n}u(\cdot,0)\right)$ is $1$-Lipschitz. The conclusion follows immediately.
    
\end{proof}

\subsection{Global lower bound for the scalar curvature}

Let us now consider the ancient solution in the statement of Proposition \ref{Prop:pinching compact}. Let $(\mathcal{X},\nu_t)$ be the tangent flow obtained in \eqref{def-limiting-condition}. Let $(\mathcal{R}_X,\fg,f_0)$ be the regular part of the model. In particular $(\mathcal{R}_{-1},\fg_{-1})=(\mathcal{R}_X,\fg)$. By the assumption of Proposition \ref{Prop:pinching compact} and Theorem \ref{thmR>}, we assume
\begin{align*}
    R_{\fg}\ge \delta>0\quad \text{ on }\quad \mathcal{R}_X=\mathcal{R}_{-1}.
\end{align*}

Pick $z_i\in M$, such that  $(z_i,-1/2)$ is an $H_n$-center of $(y_0,0)$ with respect to the rescaled flow $g^i_t$ (or, in other words, $(z_i,-\lambda_i/2)$ is an $H_n$-center with respect to the unscaled flow $g_t$). 

\begin{Lemma}\label{Lm:global scalar lower bound}
    For any $A<\infty$, we have
    \begin{align*}
        R_{g^i_{-1/2}}\ge \delta/4\qquad \text{on}\qquad B_{g^i_{-1/2}}(z_i,A)
    \end{align*}
    for all $i$ large enough.
\end{Lemma}

\begin{proof}
First of all, we show that the integral of the scalar curvature is not small if $i$ is large enough.
\\

\textbf{Claim.} \emph{For any sequence $\{x_i\}_{i=1}^\infty$ with $x_i\in B_{g^i_{-1/2}}(z_i,A)$, we have}
\begin{equation*}%\label{ineq liminf>}
        \liminf_{i\to\infty}\int_MR_{g{^i_{-1}}}d\nu_{x_i,-1/2;-1}^i\ge  \delta.
    \end{equation*}
\begin{proof}[Proof of the claim]
     For any  $x_i\in B_{g^i_{-1/2}}(z_i,A)$, we have, by  \cite[Definition 3.10]{Bam20a} 
    \begin{equation*}%\label{ineq dW1}
        d_{W_1}^{g^i_{-1/2}}(\delta_{x_i},\nu^i_{y_0,0;-1/2})\le d_{W_1}^{g^i_{-1/2}}(\delta_{x_i},\delta_{z_i})+d_{W_1}^{g^i_{-1/2}}(\delta_{z_i},\nu^i_{y_0,0;-1/2})\le A+\sqrt{\tfrac{H_n}{2}}.
    \end{equation*}
    Thus for any subsequence of $\{x_i\}_{i=1}^\infty$, according to Theorem \ref{Thm:compactness of points}, we may find a conjugate heat flow $(\tilde\nu_t)_{t<-1/2}$ on $\mathcal{X}$, such that, after passing to a further subsequence,
    \begin{align*}
       \left( \nu^i_{x_i,-1/2\,;\,t}\right)_{t\le -1/2}\xrightarrow[i\to\infty]{\mathfrak{C}} (\tilde{\nu}_{t})_{t<-1/2}.
    \end{align*}
    By Theorem \ref{Thm_smooth_convergence}(d), we have
    \begin{align*}
        \nu^i_{x_i,-1/2;-1}\xrightarrow[i\to\infty]{C_{\loc}^{\infty}}\tilde\nu_{-1}\quad \text{ on }\quad \mathcal{R}_{-1},
    \end{align*}
    where the local smooth convergence is also understood in terms of Theorem \ref{Thm_smooth_convergence}(d). Since $R_{\fg}\ge \delta$ on $\mathcal{R}_{-1}$ and $\tilde\nu_{-1}(\mathcal{S}_{-1})=0$, by Fatou's lemma we have
       \begin{equation*}
        \liminf_{i\to\infty}\int_MR_{g{^i_{-1}}}d\nu_{x_i,-1/2;-1}^i\ge \int_{\mathcal{R}_{-1}}R_{\fg}\,d\tilde\nu_{-1}\ge \delta.
    \end{equation*}
    The conclusion then follows.
\end{proof}

Let us fix $r=r(\delta)>0$, such that
\begin{align}\label{eq:choice of r}
    2n\sqrt{2}\Phi\left(\frac{1}{\sqrt{2}}\Phi^{-1}\left(\frac{\delta}{4n}\right)-\sqrt{2}r\right)>\frac{\delta}{4}.
\end{align}
   Then, for each $i$, we pick a finite sequence 
    $\left\{x_k^{(i)}\right\}_{k=1}^{N_i}$ in $B^i_{-1/2}(z_i,A)$, such that
    \begin{align*}
        B_{g^i_{-1/2}}(z_i,A)\subset\bigcup_{k=1}^{N_i} B_{g^i_{-1/2}}\left(x_k^{(i)},r\right),\qquad N_i\le\overline{N}(A,r), 
    \end{align*}
    by the Bishop-Gromov volume comparison,  where $\overline{N}$ is independent of $i$. By passing to a subsequence, we may assume $N_i\equiv N$.

Applying the claim to the sequence $\left\{x_k^{(i)}\right\}_{i=1}^\infty$ for $k=1,2,\hdots,N$, we obtain that whenever $i$ is large enough
\begin{align*}
    \int_MR_{g{^i_{-1}}}d\nu_{x^{(i)}_k,-1/2;-1}^i\ge \delta/2,\quad k=1,2,\hdots, N.
\end{align*}
Applying Lemma \ref{Lm:local bound for scalar curvature} with $r^2\to 1/2$, $\delta\to \delta/4$, $A\to \sqrt{2}r$, and in view of the choice of $r$ \eqref{eq:choice of r}, we have that, when $i$ is large enough
\begin{align*}
    R_{g^i_{-1/2}}\ge \frac{\delta}{4}\quad \text{ on }\quad B_{g^i_{-1/2}}\left(x_k^{(i)},r\right)\quad \text{ for all }\quad k=1,2,\hdots,N .
\end{align*}
    The conclusion follows immediately.
\end{proof}

\subsection{Completion of the proof}

\begin{proof}[Proof of Proposition \ref{Prop:pinching compact}]
  
By Lemma \ref{Lm:global scalar lower bound}, for any $A<\infty$, we have $\Ric_{g^i_{-1/2}}\ge\tfrac{1}{4}\varepsilon\delta$ on $B_{g^i_{-1/2}}(z_i,A)$ for all $i$ large enough, where $z_i$ is the point defined at the beginning of \S 5.2.  If we pick $A\ge 100\sqrt{\tfrac{4\pi^2(n-1)}{\varepsilon\delta}}$, then by Bonnet-Myers theorem, no minimal geodesic in $(M,g^i_{-1/2})$ emanating from $z_i$ can be longer than $A/100$. This cannot happen unless $M$ is closed.
\end{proof}

\bigskip

%\bibliography{bibliography}{}

\begin{thebibliography}{CCG{\alphalchar{+}}10}
% \bibitem[A90]{A90} Anderson, Michael T. \emph{Convergence and rigidity of manifolds under Ricci curvature bounds.} Invent. math. \textbf{102}, no. 1 (1990): 429-445.


% \bibitem[AC92]{AC92} Anderson, Michael T.; Cheeger, Jeff. \emph{$ C^\alpha $-compactness for manifolds with Ricci curvature and injectivity radius bounded below.} J. Differential Geom. \textbf{35}, no. 2 (1992): 265-281.

% \bibitem[CC96]{CC96}
% Cheeger, Jeff; Colding, Tobias H. 
% \emph{Lower bounds on Ricci curvature and the almost rigidity of warped products}.
% Ann. of Math. (2) \textbf{144} (1996), no. 1, 189–237.

% \bibitem[CC97]{CC97}
% Cheeger, Jeff; Colding, Tobias H. 
% \emph{On the structure of spaces with Ricci curvature bounded below. I}. J. Differential Geom. \textbf{46} (1997), no. 3, 406–480. 

% \bibitem[CC20a]{CC20a}
% Cheeger, Jeff; Colding, Tobias H. 
% \emph{On the structure of spaces with Ricci curvature bounded below. II}. J. Differential Geom. \textbf{54} (2000), no. 1, 13–35. 

% \bibitem[CC20b]{CC20b}
% Cheeger, Jeff; Colding, Tobias H. 
% \emph{On the structure of spaces with Ricci curvature bounded below. III}. J. Differential Geom. \textbf{54} (2000), no. 1, 37–74. 


% \bibitem[CCT02]{CCT02}
% Cheeger, Jeff; Colding, Tobias H.; Tian, Gang.
% \emph{On the singularities of spaces with bounded Ricci curvature}. Geom. Funct. Anal. \textbf{12} (2002), no. 5, 873–914. 


% \bibitem[CJN21]{CJN21}
% Cheeger, Jeff; Jiang, Wenshuai; Naber, Aaron.  \emph{Rectifiability of singular sets of noncollapsed limit spaces with Ricci curvature bounded below}. Ann. of Math. (2) \textbf{193} (2021), no. 2, 407–538.
 



% \bibitem[CN13]{CN13} Cheeger, Jeff; Naber, Aaron. \emph{Lower bounds on Ricci curvature and quantitative behavior of singular sets}. Invent. Math. \textbf{191} (2013), no. 2, 321–339.




% \bibitem[GT]{GT}
% Gilbarg, David; Trudinger, Neil S. 
% \emph{Elliptic partial differential equations of second order.} Vol. 224. Springer Science \& Business Media, 2001.


% \bibitem[HH97]{HH97} Hebey, Emmanuel; Herzlich, Marc. \emph{Harmonic coordinates, harmonic radius and convergence of Riemannian manifolds.} Rend. Mat. Appl.(7) \textbf{17}, no. 4 (1997): 569-605.

%\bibitem[ABDS22]{ABDS22}  Angenent, Sigurd; Brendle, Simon; Daskalopoulos, Panagiota; Sesum, Natasa. \emph{Unique asymptotics of compact ancient solutions to three-dimensional Ricci flow.} Comm. Pure Appl. Math. 75 (2022), no. 5, 1032-–1073.

%\bibitem[App19]{App19} Appleton, Alexander. \emph{Eguchi-Hanson singularities in $U(2)$-invariant Ricci flow}. arXiv preprint arXiv:1903.09936 (2019).

%\bibitem[Bam18]{Bam18} Bamler, Richard~H. \emph{Long-time behavior of $3$-dimensional Ricci flow -- Introduction}. Geometry \& Topology 22-2 (2018), 757--774



\bibitem[Bam20a]{Bam20a}
Bamler, Richard~H. \emph{Entropy and heat kernel bounds on a Ricci flow background}.
arXiv preprint arXiv:2008.07093 (2020).


  
\bibitem[Bam20b]{Bam20b} \bysame, \emph{Structure theory of non-collapsed limits of Ricci flows}. arXiv preprint  arXiv:2009.03243 (2020).


\bibitem[Bam23]{Bam23}\bysame, \emph{Compactness theory of the space of super Ricci flows.} Invent. Math. 233 (2023), no. 3, 1121–1277. 


%\bibitem[BCMZ21]{BCMZ21} Bamler, Richard H.; Chan, Pak-Yeung; Ma, Zilu; Zhang, Yongjia. \emph{ An optimal volume growth estimate for noncollapsed steady gradient Ricci solitons}. arXiv preprint  arXiv:2110.04661 (2021).
 
%\bibitem[BCDMZ21]{BCDMZ21} Bamler, Richard~H.; Chow, Bennett; Deng, Yuxing; Ma, Zilu; Zhang, Yongjia, \emph{ Four-dimensional steady gradient Ricci solitons with $3$-cylindrical tangent flows at infinity}. Adv. Math. 401 (2022), Paper No. 108285, 21 pp.

%\bibitem[BK22]{BK22} Bamler, Richard H.; Kleiner, Bruce. \emph{ Ricci flow and diffeomorphism groups of $3$-manifolds.} J. Amer. Math. Soc. DOI: https://doi.org/10.1090/jams/1003 Published electronically: August 12, 2022.

%\bibitem[BK21a]{BK21a} Bamler, Richard H.; Kleiner, Bruce. \emph{ Diffeomorphism groups of prime $3$-manifolds.} arXiv preprint arXiv:2108.03302 (2021).

%\bibitem[BK21b]{BK21b} Bamler, Richard H.; Kleiner, Bruce. \emph{On the rotational symmetry of $3$-dimensional $\kappa$-solutions.} J. Reine Angew. Math. 779 (2021), 37-–55.

%\bibitem[Bre13]{Bre13} Brendle, Simon. \emph{Rotational symmetry of self-similar solutions to the Ricci flow}. Inventiones Mathematicae 194, 731--764 (2013).

\bibitem[Bre09]{Bre09} Brendle, Simon. \emph{A generalization of Hamilton's differential Harnack inequality for the Ricci flow}. J. Differential Geom. 82(1): 207-227. DOI: 10.4310/jdg/1242134372

%\bibitem[Bre20]{Bre20} Brendle, Simon. \emph{Ancient solutions to the Ricci flow in dimension $3$}. Acta Mathematica 225, 1--102 (2020).

%\bibitem[BDS21]{BDS21} Brendle, Simon; Daskalopoulos, Panagiota; Sesum, Natasa. \emph{Uniqueness of compact ancient solutions to three-dimensional Ricci flow.} Invent. Math. 226 (2021), no. 2, 579–651.

\bibitem[Chow23]{Chow23} Chow, Bennett; Ricci solitons in low dimensions, GTM235


%\bibitem[CN09]{CN09} Carrillo, Jos\'e A., and Ni,Lei. \emph{Sharp logarithmic Sobolev inequalities on gradient solitons and applications.} Communications in Analysis and Geometry 17.4 (2009): 721-753.

\bibitem[CZ11]{CZ11} Cao, Xiaodong; Zhang, Qi S.  \emph{The conjugate heat equation and ancient solutions of the Ricci flow.} Adv. Math. \textbf{228} (2011), no. 5, 2891–2919.

\bibitem[CLY11]{CLY11} Chow, Bennett; Lu, Peng; Yang, Bo. \emph{ Lower bounds for the scalar curvatures of noncompact gradient Ricci solitons}. C. R. Math. Acad. Sci. Paris 349 (2011), no. 23-24, 1265–1267.

\bibitem[CLY12]{CLY12}Chow, Bennett; Lu, Peng; Yang, Bo. \emph{A necessary and sufficient condition for Ricci shrinkers
to have positive AVR}. Proc. Amer. Math. Soc. 140 (2012), no. 6, 2179–2181, DOI 10.1090/S0002-9939-2011-11173-0. MR2888203

\bibitem[CMZ25]{CMZ25} Chan, Pak-Yeung; Ma, Zilu; Zhang, Yongjia. \emph{Ancient Ricci flows with asymptotic solitons.}J Geom Anal 35, 199 (2025). https://doi.org/10.1007/s12220-025-02034-z

%\bibitem[CMZ22]{CMZ22} Chan, Pak-Yeung; Ma, Zilu; Zhang, Yongjia. \emph{A local gap theorem for Ricci shrinkers.} arXiv preprint 	arXiv:2212.09203 (2022).

%\bibitem[CMZ23]{CMZ23} Chan, Pak-Yeung; Ma, Zilu;  Zhang, Yongjia. \emph{A uniform Sobolev inequality for ancient Ricci flows with bounded Nash entropy.} International Mathematics Research Notices 2023.13 (2023): 11127-11144.

%\bibitem[CMZ23a]{CMZ23a} Chan, Pak-Yeung; Ma, Zilu; Zhang, Yongjia. \emph{A local Sobolev inequality on Ricci flow and its applications.} J. Funct. Anal. 285 (2023), no. 5, Paper No. 109995, 36 pp.

\bibitem[CMZ24]{CMZ23} Chan, Pak-Yeung; Ma, Zilu; Zhang, Yongjia. \emph{On noncollapsed $\bF$-limit of metric solitons.} arXiv preprint arXiv:2401.03387 (2023)

%\bibitem[CCMZ23]{CCMZ23} Chan, Pak-Yeung; Chow, Bennett; Ma, Zilu;  Zhang, Yongjia. \emph{ Lower bounds for the scalar curvatures of Ricci flow singularity models.} Journal f\"ur die reine und angewandte Mathematik (Crelles Journal), 2023(794), 253-265.

%\bibitem[CN15]{CN15} Cheeger, Jeff; Naber, Aaron. \emph{Regularity of Einstein manifolds and the codimension }$4$\emph{ conjecture}.   Ann. of Math. (2) \textbf{182} (2015), no. 3, 1093–1165.

%\bibitem[Che09]{Che09}Chen, Bing-Long,\emph{Strong uniqueness of the Ricci flow}. J. Differential Geom. 82 (2009), 363--382.


%\bibitem[CCGGIIKLLN10]{RFV3} Chow, B.; Chu, S.; Glickenstein, D.; Guenther, C.; Isenberg, J.; Ivey, T.; Knopf, D.; Lu, P.; Luo, F.; Ni, L. \emph{The Ricci flow: techniques and applications. Part III. Geometric-Analytic Aspects}, Mathematical Surveys and Monographs,  vol. \textbf{163}, AMS, Providence, RI, 2010.

%\bibitem[CFSZ20]{CFSZ20} Chow, Bennett; Freedman, Michael; Shin, Henry; Zhang, Yongjia \emph{ Curvature growth of some $4$-dimensional gradient Ricci soliton singularity models}. Adv. Math. 372 (2020), 107303, 17 pp.


%\bibitem[CLY11]{CLY11} Chow, Bennett; Lu, Peng; Yang, Bo \emph{ Lower bounds for the scalar curvatures of noncompact gradient Ricci solitons}. C. R. Math. Acad. Sci. Paris 349 (2011), no. 23-24, 1265–1267. 

%\bibitem[CCGGIIKLLN07]{RFV1}Chow, B.; Chu, S.; Glickenstein, D.; Guenther, C.; Isenberg, J.; Ivey, T.; Knopf, D.; Lu, P.; Luo, F.; Ni, L. \emph{The Ricci flow: techniques and applications. Part I. Geometric Aspects}, Mathematical Surveys and Monographs,  vol. \textbf{135}, AMS, Providence, RI, 2007.

%\bibitem[CCGGIIKLLN08]{RFV2}Chow, B.; Chu, S.; Glickenstein, D.; Guenther, C.; Isenberg, J.; Ivey, T.; Knopf, D.; Lu, P.; Luo, F.; Ni, L. \emph{The Ricci flow: techniques and applications. Part II. Analytic Aspects}, Mathematical Surveys and Monographs,  vol. \textbf{144}, AMS, Providence, RI, 2008.






%\bibitem[CCGGIIKLLN15]{RFV4} Chow, B.; Chu, S.; Glickenstein, D.; Guenther, C.; Isenberg, J.; Ivey, T.; Knopf, D.; Lu, P.; Luo, F.; Ni, L. \emph{The Ricci Flow: Techniques and Applications. Part IV: Long-Time Solutions and Related Topics}, Mathematical Surveys and Monographs, vol. \textbf{206}, AMS, Providence, RI, 2015.
 
\bibitem[CZ10]{CZ10}Cao, Huai-Dong; Zhou, De-Tang. \emph{On complete gradient shrinking Ricci solitons.} J. Differential Geom. \textbf{85} (2010), 175--186. 

% \bibitem[CLN06]{CLN06}  Bennett Chow, Peng Lu and Lei Ni, \emph{Hamilton's Ricci flow}, Lectures in Contemporary Mathematics, \textbf{3}, Science Press and Graduate Studies in Mathematics, \textbf{77}, American Mathematical Society (co-publication), 2006.

%\bibitem[DZ18]{DZ18}  Deng, Yuxing; Zhu, Xiaohua. \emph{Asymptotic behavior of positively curved steady Ricci solitons}, Trans. Amer. Math. Soc. \textbf{370} (2018), 2855-2877.

\bibitem[DZ15]{DZ15} Deng, Yuxing; Zhu, Xiaohua. \emph{Complete non-compact gradient Ricci solitons with nonnegative Ricci curvature}.
Math. Z. 279 (2015), 211–226.

%\bibitem[DZ20a]{DZ20a} Deng, Yuxing; Zhu, Xiaohua. \emph{Higher dimensional steady Ricci solitons with linear curvature decay.} J. Eur. Math. Soc. (JEMS) \textbf{22} (2020), no. 12, 4097–4120.

%\bibitem[DZ20b]{DZ20b} Deng, Yuxing; Zhu, Xiaohua. \emph{Classification of gradient steady Ricci solitons with linear curvature decay}, Sci. China Math. \textbf{63} (2020), no. 1, 135–154.

\bibitem[FL25]{FL25} Fang, Hanbing; Li, Yu. \emph{On the structure of noncollapsed Ricci flow limit spaces}, arXiv preprint 	arXiv:2510.12398(2025)


%\bibitem[FIK03]{FIK03} Feldman, Mikhail,  Ilmanen, Tome, and Knopf, Dan. \emph{Rotationally symmetric shrinking and expanding gradient K\"ahler-Ricci solitons}. Journal of Differential Geometry 65.2 (2003): 169-209.

\bibitem[Ham82]{Ham82} Hamilton, Richard S, \emph{Three-manifolds with positive Ricci curvature}. Journal of Differential geometry 17.2 (1982): 255-306.

%\bibitem[Ham88]{Ham88} Hamilton, Richard S.  \emph{The Ricci flow on surfaces}. Mathematics and general relativity (Santa Cruz, CA, 1986), 237--262, Contemp. Math., 71, Amer. Math. Soc., Providence, RI, 1988.


%\bibitem[Ham93a]{Ham93a} Richard S. Hamilton,  \emph{The Harnack estimate for the Ricci flow}, J. Differential Geom., 1993, 37(1): 225-243.

\bibitem[Ham93]{Ham93} Hamilton, Richard S.  \emph{The formation of singularities in the Ricci flow.\ }Surveys in differential geometry, Vol.\ II (Cambridge, MA, 1993), 7--136, Internat. Press, Cambridge, MA, 1995.

\bibitem[Ham94]{Ham94} Hamilton, Richard S.  \emph{Convex hypersurfaces with pinched second fundamental form.} Comm. Anal. Geom.
24(1994), vol 2., no.1, 167–72.

%\bibitem[Ham95]{Ham95} Hamilton, Richard S. \emph{The formation of singularities in the Ricci flow}, Surveys in Differential Geometry, \textbf{2} (1995): 7-136

%\bibitem[Han20]{Han20} Han, Daoyuan. \emph{ Asymptotic curvature estimate for steady solitons}. arXiv preprint arxiv:2009.04665 (2020).
\bibitem[HM11]{HM11} Haslhofer, Robert; M\"uller, Reto. \emph{A compactness theorem for complete Ricci shrinkers}, Geom.
Funct. Anal. 21 (2011), no. 5, 1091–1116, DOI 10.1007/s00039-011-0137-4. MR2846384

% \bibitem[KL08]{KL08} Bruce Kleiner and John Lott, \emph{Notes on Perelman's papers}, Geom. Topol.\textbf{12} (2008), no. 5, 2587--2855.

%\bibitem[Kot13]{Kot13} Kotschwar, Brett.\emph{A local version of Bando's theorem on the real-analyticity of solutions to the Ricci flow}. Bull. Lond. Math. Soc. 45 (2013), no.1, 153–158.

%\bibitem[Lai20]{Lai20} Lai, Yi. \emph{A family of 3d steady gradient solitons that are flying wings.} arXiv preprint arXiv:2010.07272 (2020), J. Differential Geom., to appear.

%\bibitem[Lai22]{Lai22}Lai, Yi. \emph{3D flying wings for any asymptotic cones}. arXiv preprint  arXiv:2207.02714 (2022).

%\bibitem[Lee13]{Lee13}Lee, John M. \emph{Introduction to smooth manifolds}, Second edition, Grad. Texts in Math., 218 Springer, New York, 2013.

\bibitem[LT25]{LT25}Lee, Man-Chun; Topping, Peter M. \emph{Three-manifolds with non-negatively pinched Ricci curvature.} J. Differential Geom. 131 (3) 633 - 651, November 2025. https://doi.org/10.4310/jdg/1760725985

%\bibitem[LW20]{LW20} Li, Yu; Wang, Bing. \emph{Heat kernel on Ricci shrinkers}. Calc. Var. Partial Differential Equations 59 (2020), no. 6, Paper No. 194, 84 pp.

\bibitem[LW24]{LW24}Li, Yu; Wang, Bing. \emph{Heat kernel on Ricci shrinkers (II).} Acta Mathematica Scientia 44.5 (2024): 1639-1695.

%\bibitem[MZ21]{MZ21} Ma, Zilu; Zhang, Yongjia. \emph{  Perelman's entropy on ancient Ricci flows}. J. Funct. Anal. 281 (2021), no. 9, Paper No. 109195, 31 pp.

%\bibitem[M20]{M20} Boris Mityagin, \emph{The Zero Set of a Real Analytic Function}. (Russian) Mat. Zametki 107 (2020), no.3, 473–475; translation in Math. Notes 107 (2020), no.3-4, 529–530. 	
%arXiv preprint, arXiv:1512.07276.

%\bibitem[MSW19]{MSW19} Munteanu, Ovidiu; Sung, Chiung-Jue Anna; Wang, Jiaping. \emph{  Poisson equation on complete manifolds}. Adv. Math. 348 (2019), 81–-145.

%\bibitem[MW12]{MW12} Ovidiu Munteanu, and Jiaping Wang, \emph{Analysis of weighted Laplacian and applications to Ricci solitons}, Communications in Analysis and Geometry 20, no. 1 (2012): 55-94.

%\bibitem[MW19]{MW19} Munteanu, Ovidiu; Wang, Jiaping. \emph{ Structure at infinity for shrinking Ricci solitons}. Annales Scientifiques de l'Ecole Normale Superieure 52 (2019), 891--925.

\bibitem[Nab10]{Nab10}Naber, Aaron. \emph{Noncompact shrinking four solitons with nonnegative curvature.} Journal f\"ur die Reine und Angewandte Mathematik 645 (2010): 125-153.

\bibitem[Ni05]{Ni05} Ni, Lei. \emph{Ancient solutions to K\"ahler Ricci Flow}. Mathematical Research Letters 12 (2005), 633–654

\bibitem[Ni25]{Ni25} Ni, Lei.  \emph{Is a complete Riemannian manifold with positively pinched Ricci curvature compact} arXiv preprint arXiv:2510.10279

\bibitem[Per02]{Per02} Perelman, Grisha, \emph{The entropy formula for the Ricci flow and its geometric applications}, arXiv:math.DG/0211159 (2002).

\bibitem[Per03a]{Per03a} Perelman, Grisha, \emph{Ricci flow with surgery on three-manifolds}, arXiv:math.DG/0303109 (2003).

\bibitem[Per03b]{Per03b} Perelman, Grisha. \emph{Finite extinction time for the solutions to the Ricci flow on certain three-manifolds}, arXiv preprint math/0307245 (2003).

%\bibitem[SS72]{SS72} Sou\v{c}ek, Ji\v{r}\'{\i}; Sou\v{c}ek, Vladim\'{\i}r. \emph{Morse-Sard theorem for real-analytic functions}. Commentationes Mathematicae Universitatis Carolinae, Vol. 13 (1972), No. 1, 45-51

%\bibitem[W18]{W18} Wang, Bing, \emph{The local entropy along Ricci flow Part A: the no-local-collapsing theorems.} Camb. J. Math. \textbf{6} (2018), no. 3, 267–346.

%\bibitem[Yo09]{Y09} Yokota, Takumi. \emph{Perelman's reduced volume and a gap theorem for the Ricci flow.} Communications in Analysis and Geometry 17.2 (2009): 227-263.

%\bibitem[Yo12]{Y12} Yokota, Takumi. \emph{Addendum to `Perelman’s reduced volume and a gap theorem for the Ricci flow'.} Communications in Analysis and Geometry 20.5 (2012): 949-955.

%\bibitem[Zh10]{Zh10} Zhang, Qi S. \emph{Sobolev inequalities, heat kernels under Ricci flow, and the Poincaré conjecture.} CRC Press, 2010.

\bibitem[Zhang11]{Zhang11} Zhang, Shijin; \emph{On a sharp volume estimate for gradient Ricci solitons with scalar curvature bounded below,} Acta Math. Sin. (Engl. Ser.) 27(2011), no. 5, 871–882, DOI 10.1007/s10114011-9527-7. MR2786449

\bibitem[Zhang20]{Zhang20} Zhang, Yongjia; \emph{On the equivalence between noncollapsing and bounded entropy for ancient solutions to the Ricci flow}, Journal für die reine und angewandte Mathematik (Crelle's Journal), 762(2020), 35-51.


\end{thebibliography}
\bibliographystyle{amsalpha}

\newcommand{\alphalchar}[1]{$^{#1}$}
\providecommand{\bysame}{\leavevmode\hbox to3em{\hrulefill}\thinspace}
\providecommand{\MR}{\relax\ifhmode\unskip\space\fi MR }
\providecommand{\MRhref}[2]{%
  \href{http://www.ams.org/mathscinet-getitem?mr=#1}{#2}
}

\noindent School of Mathematics and Statistics, Beijing Institute of Technology, Beijing, 100081, China
\\ E-mail address: \verb"6120180026@bit.edu.cn"
\\

\noindent School of Mathematics and Statistics, Beijing Institute of Technology, Beijing, 100081, China
\\ E-mail address: \verb"ggwgq5986@sina.com"
\\

\noindent School of Mathematical Sciences, Shanghai Jiao Tong University, Shanghai, 200240, China
\\ E-mail address: \verb"sunzhang91@sjtu.edu.cn"

\end{document}